\documentclass[reqno]{amsart}

\usepackage[dvipsnames]{xcolor}
\usepackage[margin=3cm]{geometry}

\usepackage{lmodern}

\usepackage[utf8]{inputenc}
\usepackage[T1]{fontenc}
\usepackage{textcomp}
\usepackage{amsmath, amssymb,mathrsfs,amsthm}
\usepackage{mathtools}
\usepackage{multirow}

\usepackage{tikz}
\usetikzlibrary{arrows}

\usepackage[linktocpage]{hyperref}
\hypersetup{	
  colorlinks=true,
  linkcolor=NavyBlue,
  anchorcolor=Blue,
  citecolor=Green,
  urlcolor=NavyBlue
}

\numberwithin{equation}{section}

\renewcommand{\d}{\,\mathrm{d}}

\newcommand{\un}{\ensuremath{\mathbf{1}}}

\newcommand\N{\ensuremath{\mathbb{N}}}
\newcommand\R{\ensuremath{\mathbb{R}}}
\newcommand\Z{\ensuremath{\mathbb{Z}}}

\newcommand\C{\ensuremath{\mathbb{C}}}

\newcommand\qc{\ensuremath{\mathscr{Q}}}
\newcommand\cc{\ensuremath{\mathcal{C}}}
\newcommand\dc{\ensuremath{\mathscr{D}}}

\newcommand{\ffi}{\varphi}
\newcommand{\eps}{\varepsilon}

\newtheorem{theorem}{Theorem}[section]
\newtheorem{lemma}[theorem]{Lemma}
\newtheorem{proposition}[theorem]{Proposition}

\newtheorem{corollary}[theorem]{Corollary}
\newtheorem{assumptionx}{Assumption}
\newenvironment{assumption}[1]
 {\renewcommand\theassumptionx{#1}\assumptionx}
 {\endassumptionx}
 
 \theoremstyle{definition}
 \newtheorem{remark}[theorem]{Remark}
 \newtheorem{definition}[theorem]{Definition}

\DeclareMathOperator{\Realpart}{Re}
\renewcommand{\Re}{\Realpart}

\newcommand{\norm}[1]{{\left\|{#1}\right\|}}
\newcommand{\abs}[1]{{\left|{#1}\right|}}
\newcommand{\scal}[1]{{\left\langle{#1}\right\rangle}}

\mathtoolsset{showonlyrefs}

\author{Yann Bourroux}
\address{Facultad de Ciencia y Tecnolog\'ia, Universidad del Pa\'is Vasco /Euskal Herriko Unibertsitatea (UPV/EHU), Departamento de Matem\'aticas, UPV/EHU, Apartado 644, 48080 Bilbao, Spain\\ \& Univ. Bordeaux, CNRS, Bordeaux INP, IMB, UMR 5251, F-33400 Talence, France}
\email{yann.bourroux@math.u-bordeaux.fr}

\author{Philippe Jaming}
\address{Univ. Bordeaux, CNRS, Bordeaux INP, IMB, UMR 5251, F-33400 Talence, France}
\email{philippe.jaming@math.u-bordeaux.fr}

\author{Yunlei Wang}
\address{Univ. Bordeaux, CNRS, Bordeaux INP, IMB, UMR 5251, F-33400 Talence, France and
Department of Mathematics, Louisiana State University, Baton Rouge, LA, 70803, USA (current address)}
\email{ywang30@lsu.edu}



\title[Controllability for semi-discrete heat equations]{Approximate null-controllability of discrete heat equations with potentials on lattices}

\begin{document}
\begin{abstract}
We investigate approximate null-controllability for semi-discrete heat equations on the lattice $h\mathbb Z^d$ with a potential. By establishing spectral inequalities for the discrete Schrödinger operator $P_h=-\Delta_h+V$ on equidistributed sets, we derive observability estimates via the Lebeau--Robbiano method and the Hilbert Uniqueness Method. 

For bounded potentials, we obtain quantitative controllability results with explicit dependence on the potential and show near optimality of the geometric condition on the observation set. We also treat polynomial growth potentials, for which similar properties hold with weaker control cost estimates. These results extend discrete Carleman techniques to the full-space lattice setting and provide new spectral estimates for discrete Schrödinger operators.
\end{abstract}

\maketitle

\tableofcontents

\section{Introduction}
In this paper we study controllability properties of the heat equation on the lattice $h\mathbb Z^d$ in the presence of a potential. More precisely, we consider semi-discrete heat equations associated with the discrete Schrödinger operator 
$$
P_h=-\Delta_h+V
$$
and investigate approximate null-controllability from sets that are equidistributed in space. Our main result shows that when the potential is bounded or has at most polynomial growth, approximate null-controllability holds for sufficiently small mesh size, with an explicit approximation error depending on $h$. 

Let us describe the setting. Let $h>0$ denote the mesh size and consider the standard discretization of the Laplacian on $\mathbb R^d$. For a function $u:h\mathbb Z^d\to\mathbb R$ and $x\in h\mathbb Z^d$, we define
\[
\Delta_h u(x)=\frac{1}{h^2}\sum_{j=1}^d\bigl(u(x+he_j)+u(x-he_j)-2u(x)\bigr),
\]
where $(e_j)_{j=1,\dots,d}$ denotes the canonical basis of $\mathbb R^d$. Equivalently, $-h^2\Delta_h$ coincides with the graph Laplacian on the lattice $h\mathbb Z^d$. Given a real-valued potential $V:\mathbb R^d\to\mathbb R$, we study approximate null-controllability properties for the semi-discrete heat equation
\[
\partial_t u=-P_hu:=\Delta_h u - Vu.
\]

Understanding controllability properties of discretized parabolic equations is a central issue both for numerical analysis and control theory (see e.g. \cite{boyer} and references therein). In contrast with the continuous setting, discretization may introduce spurious high-frequency phenomena and weaken unique continuation properties, which makes uniform controllability with respect to the mesh size a delicate question. 

While semi-discrete controllability has been extensively studied on bounded domains, much less is known in the full-space setting, especially in the presence of lower-order terms. The purpose of this paper is to address this question for discrete Schrödinger operators on lattices.

\smallskip

\paragraph{\bf Main contributions.} 
\begin{itemize}
    \item We prove a spectral inequality for discrete Schrödinger operators on $h\Z^d$
 with bounded or polynomial growth potentials over equidistributed sets.

 \item Using the Lebeau–Robbiano method, we deduce approximate null-controllability results with explicit error bounds uniform with respect to the mesh size.

 \item Our analysis extends the semi-discrete Carleman framework of Boyer–Hubert–Le Rousseau to the full-space setting and to operators with potential.

 \item In the bounded potential case, we obtain explicit decay of the controllability error
 and show that the condition on the controllability set is essentially optimal.
\end{itemize}

Our approach relies on quantitative unique continuation estimates obtained through Carleman inequalities. More precisely, we adapt the semi-discrete Carleman framework introduced by Boyer, Hubert, and Le~Rousseau to derive spectral inequalities for the operator $P_h$. These spectral estimates are then combined with the Lebeau--Robbiano method and the Hilbert Uniqueness Method to establish observability and controllability results. The main difficulty lies in handling the presence of the potential in the full-space discrete setting, where the spectral parameter interacts nontrivially with the discretization scale.

\smallskip

\subsection*{Relation to the literature.}\ \\
Null-controllability and observability properties for parabolic equations have been extensively studied in the continuous setting. In particular, quantitative unique continuation estimates and spectral inequalities play a central role through the Lebeau--Robbiano method. For the heat equation on $\mathbb R^d$, null-controllability from measurable sets is closely related to geometric thickness conditions \cite{EV,WWZZ}. More generally, spectral inequalities and propagation estimates for Schrödinger operators with potentials have been developed in a series of recent works, see for instance \cite{burq2022propagation,burq2025propagation,yunlei2025quantitative,zhu2024spectral,lebalch2025quantitative,malinnikova2025spectral,huang2024obs} and references therein.

In the semi-discrete setting, uniform controllability with respect to the mesh size was initiated by Boyer, Hubert, and Le~Rousseau \cite{boye2010disc,boye2010discr2}, who developed discrete Carleman estimates to prove approximate null-controllability for parabolic equations on bounded domains. Their results highlight the coupling between the spectral parameter and the discretization scale and establish partial spectral inequalities for low-frequency components. Earlier works also considered boundary control problems in one dimension and highlighted limitations of uniform controllability in higher dimensions \cite{lopez1998some}.

In contrast, much less is known in the full-space discrete setting. A recent work \cite{Wan2025obs} studies observability for discrete Schrödinger equations on combinatorial graphs and reveals phenomena that differ significantly from the continuous case, including threshold effects related to thickness conditions. However, controllability results for discrete heat equations on lattices with lower-order terms remain largely unexplored.

The present work extends the semi-discrete Carleman approach of \cite{boye2010disc,boye2010discr2} to discrete Schrödinger operators on the full lattice and establishes spectral inequalities and approximate null-controllability in the presence of bounded or polynomial growth potentials. In this way, our results connect the bounded-domain semi-discrete theory with recent advances on spectral inequalities for Schrödinger operators in the continuous full-space setting.

\subsection*{Results}\ 

\smallskip

\paragraph{\sl Spectral estimates.}
Our approach to approximate null-controllability follows the Lebeau--Robbiano strategy and is based on establishing suitable spectral inequalities for the discrete Schrödinger operator $P_h=-\Delta_h+V$. These inequalities provide quantitative control of low-frequency components from observations on equidistributed sets and constitute the main analytical ingredient of the paper.

We recall that in \cite{rojas2013}, a set $\omega\subset\mathbb R^d$ is said to be \emph{equidistributed} if there exist $0<\ell<L$ such that
each cube $kL+[-L/2,L/2]^d$, $k\in\mathbb Z^d$, contains a cube $z_k+[-\ell/2,\ell/2]^d$.
Let $\Pi_{\mu,h}$ denote the spectral projector associated with $P_h$, defined by
\[
\Pi_{\mu,h} u := \mathbf{1}_{P_h\le \mu} u = \int_{-\infty}^{\mu} \mathrm{d} m_\lambda\, u,
\]
where $m_\lambda$ denotes the spectral measure of $P_h$. Our first main result establishes the following spectral inequality.

\begin{theorem}\label{th:main:1}
    Let $\omega$ be equidistributed, $V\,:\R^d\to\R^+$ a non-negative function and let $P_h=-\Delta_h+V$ with associated spectral projector $\Pi_{\mu,h}$.
    
    \smallskip
    
    --- If $V\in \cc_b(\mathbb R^d,\R^+)$, there exist constants $C,\kappa>0$, $\varepsilon_0>0$, and $h_0>0$ depending only on $\omega$ such that for any
    $h\le h_0(1+\|V\|_{L^\infty}^{2/3})^{-1}$ and $0<\mu\le \varepsilon_0/h^2$,
    \begin{equation}
    \label{ineq-bounded-only}
        \|\Pi_{\mu,h} u\|_{\ell^2(h\mathbb Z^d)}^2
        \le C e^{\kappa \sqrt{1+\|V\|_{L^\infty}^{4/3}+\mu}} 
        \|\Pi_{\mu,h} u\|_{\ell^2(\omega)}^2,
        \quad \forall u\in \ell^2(h\mathbb Z^d).
    \end{equation}
    
    --- If $V\in \cc_b^1(\mathbb R^d,\R^+)$, there exist constants $C>0$, $\varepsilon_0>0$, and $h_0>0$ depending only on $\omega$ such that for any
    $h\le h_0(1+\|V\|_{W^{1,\infty}}^{1/2})^{-1}$ and $0<\mu\le \varepsilon_0/h^2$,
    \begin{equation}
    \label{inequ-c1-bounded}
        \|\Pi_{\mu,h} u\|_{\ell^2(h\mathbb Z^d)}^2
        \le C e^{\kappa\sqrt{1+\|V\|_{W^{1,\infty}}+\mu}} 
        \|\Pi_{\mu,h} u\|_{\ell^2(\omega)}^2,
        \quad \forall u\in \ell^2(h\mathbb Z^d).
    \end{equation}

    --- If $V(x)=\scal{x}^\beta$, there exist constants $C>0$, $\varepsilon_0>0$, and $h_0>0$ depending only on $\omega$ and $\beta$ such that for any $h<h_0$ and $0<\mu<\varepsilon_0/h^2$,
    \begin{equation*}
        \|\Pi_{\mu,h} u\|_{\ell^2(h\mathbb Z^d)}^2
        \le C e^{C\mu^{1/2}}
        \|\Pi_{\mu,h} u\|_{\ell^2(\omega)}^2,
        \quad \forall u\in \ell^2(h\mathbb Z^d).
    \end{equation*}
\end{theorem}

Here $\scal{x}=(1+|x|^2)^{1/2}$ is the usual Japanese bracket. Throughout this introduction, to avoid technicalities, we limit the statements for potentials with power growth
to this simplest case of polynomial-growth potentials. In the body of the paper we treat a broader class of polynomial-growth potentials, satisfying Assumption~\ref{assump} below.

The key feature of these estimates is that the exponential weight involves $\mu^{1/2}$, which is sublinear in the spectral parameter and therefore sufficient to derive observability inequalities through the Lebeau--Robbiano method. Another important aspect is the explicit dependence of the constants on the potential in the bounded cases, which allows us to obtain quantitative bounds on the control cost.

A fundamental difference with the continuous setting is that the spectral inequality holds only for spectral subspaces corresponding to eigenvalues below a threshold proportional to $h^{-2}$. This restriction is intrinsic to the discrete framework, as extending the inequality to arbitrarily high frequencies would imply strong unique continuation properties that fail for discrete heat equations. This limitation plays a central role in the analysis of controllability and is discussed further below.

The proof adapts the discrete Carleman estimates of \cite{boye2010disc,boye2010discr2} to incorporate the potential and to treat the full-space lattice setting.

\smallskip

\paragraph{\sl Approximate null-controllability.}
Building on the spectral inequalities established above, we derive approximate null-controllability properties
via the Lebeau–Robbiano method.
We consider the semi-discrete heat equation associated with the operator $P_h=-\Delta_h + V $,
\begin{equation}\label{eq:heat}
\begin{cases}
\partial_t u +P_hu  = \mathbf{1}_\omega f, & (t,x)\in (0,\infty)\times h\mathbb Z^d,\\
u(0,x)=u_0, & u_0\in \ell^2(h\mathbb Z^d),
\end{cases}
\end{equation}
and will prove our main result, which shows
approximate null-controllability with quantitative estimates uniform with respect to the mesh size:

\begin{theorem}\label{thm-uniform-control}
Let $V\in C_b(\mathbb R^d)$ be real non-negative valued and let $\omega$ be equidistributed.
Then for every $0<\eps\le 1$, there exist $K_\eps>0$ and $h_\eps>0$ such that the following holds.

\begin{enumerate}
\item If $T>K_\eps(1+\|V\|_{L^\infty}^{2/3})$, then there 
exists a constant $C>0$ depending only on $\omega$, $V$ and $\eps$, such that for all $0<h\le h_\eps(1+\|V\|_{L^\infty}^{2/3})^{-1}$ and all initial data $u_0\in \ell^2(h\mathbb Z^d)$, there exists a control $f$ satisfying
\[
\| f \|_{L^2((0,T);\ell^2(\omega\cap h\Z^d))} \le \frac{C}{\sqrt{T}} \|u_0\|_{\ell^2(h\mathbb Z^d)},
\]
such that the corresponding solution to \eqref{eq:heat} satisfies
\begin{equation}\label{uniform-ctrl-2}
\|u(T)\|_{\ell^2(h\mathbb Z^d)}
\le C_\varepsilon e^{-C T/h^2}\| u_0 \|_{\ell^2(h\mathbb Z^d)}.
\end{equation}
Moreover, the low-frequency components are driven to zero at time $T$, namely for $\mu=\epsilon_0/h^2$,
\begin{equation}\label{uniform-ctrl-1}
\Pi_{\mu,h}u(T)=0.
\end{equation}

\item If $T<K_\eps (1+\|V\|_{L^\infty}^{2/3})$, the same conclusion holds under the stronger condition
\[
0<h\le h_\eps \left(\frac{T}{1+\|V\|_{L^\infty}^{2/3}}\right)^{1+\epsilon}
\]
with the control cost estimate
\[
\| f \|_{L^2((0,T);\ell^2(\omega\cap h\Z^d))}
\le  C\exp\left(C\frac{(1+\|V\|_{L^\infty}^{2/3})^{2(1+\varepsilon)}}{T^{1+\eps}}\right)
\|u_0\|_{\ell^2(h\mathbb Z^d)}.
\]
\end{enumerate}

If $V\in C_b^1(\mathbb R^d)$, the above statements hold with $\|V\|_{L^\infty}^{2/3}$ replaced by $\|V\|_{W^{1,\infty}}^{1/2}$.
\end{theorem}
\begin{remark}
    By tracing the dependence of the constants on $\varepsilon$ in the proof, we find that, in the second part of the theorem, they satisfy the asymptotic relations
    \begin{equation*}
        K_\varepsilon\asymp  \varepsilon^{-1}, h_\varepsilon\asymp \varepsilon^{1+\varepsilon}\asymp \varepsilon.
    \end{equation*}
    Moreover, the control cost is bounded by
    \[
\| f \|_{L^2((0,T);\ell^2(\omega\cap h\Z^d))}
\le  \frac{C'}{ \varepsilon \sqrt{T}}\exp\left(C'\frac{(1+\|V\|_{L^\infty}^{2/3})^{2(1+\varepsilon)}}{T^{1+\eps}}\right)
\|u_0\|_{\ell^2(h\mathbb Z^d)}
\]
for some positive constant $C'$ independent of $\varepsilon$.
\end{remark}

When the potential has polynomial growth, a similar approximate null-controllability result holds, although with less precise bounds on the control cost.

\begin{theorem}\label{thm-unbounded-uniform-control}
Let $V(x)=\scal{x}^\beta$ and let $\omega\subset\mathbb R^d$ be equidistributed. There exist constants $C_0,C_1,C_2,C_3$ and $h_0$, depending only on $\omega$, $T$, and $\beta$, such that for all $0<h\le h_0$ and all initial data $u_0\in \ell^2(h\mathbb Z^d)$, there exists a control $f$ satisfying
\[
\| f \|_{L^2((0,T);\ell^2(h\mathbb Z^d))} \le C_0 \|u_0\|_{\ell^2(h\mathbb Z^d)},
\]
such that the solution of \eqref{eq:heat} satisfies
\[
\|u(T)\|_{\ell^2(h\mathbb Z^d)}
\le C_1 e^{-C_2/h^2}\| u_0 \|_{\ell^2(h\mathbb Z^d)}.
\]
Moreover, the low-frequency components are driven to zero, namely for $\mu\le C_3 h^{-2}$,
\[
\Pi_{\mu,h}u(T)=0.
\]
\end{theorem}

These results show that approximate null-controllability holds uniformly with respect to the discretization parameter, with an exponentially small remainder corresponding to high-frequency components. The estimates also provide explicit bounds on the control cost and highlight the role of the potential and the geometry of the observation set.

\smallskip

\paragraph{\sl Observability inequalities.}
Through the Hilbert Uniqueness Method (see, e.g., \cite{Li1,Li2,tucsnak2009observation}), 
these approximate null-controllability properties 
are equivalent to observability inequalities for the corresponding parabolic equation. More precisely, 
we obtain the following relaxed (uniform) observability result.

\begin{theorem}\label{th:obsineqintro}
Let $V\in \cc_b(\mathbb R^d)$ be real non-negative valued and let $\omega$ be equidistributed.
Then for every $0<\eps\le 1$, there exists $K_\eps>0$ such that the semi-discrete solution $v$ of the adjoint system
\begin{equation*}
\begin{cases}
\partial_t v+\Delta_h v -Vv=0,& \text{in } (0,T)\times h\mathbb Z^d,\\
v(T)=v_F\in \ell^2(h\mathbb Z^d),
\end{cases}
\end{equation*}
satisfies the uniform observability estimate
\begin{equation}\label{eq:necessaryintro}
\|v(0)\|_{\ell^2(h\mathbb Z^d)}^2
\le K_T \int_0^T \|v(t)\|_{\ell^2(\omega)}^2\,\mathrm{d}t
+ C e^{-C/h^2}\|v_F\|_{\ell^2(h\mathbb Z^d)}^2,
\end{equation}
where
\[
K_T=
\begin{cases}
\dfrac{C_\eps}{\sqrt{T}} & \text{if } T>K_\eps(1+\|V\|_{L^\infty}^{2/3}),
\text{ and }h\le \dfrac{h_0}{1+\|V\|_{L^\infty}^{2/3}};\\[6pt]
\dfrac{C_\eps}{\sqrt{T}}\exp\left(C_\eps\dfrac{(1+\|V\|_{L^\infty}^{2/3})^2}{T}\right)^{1+\varepsilon}
& \text{if } T\le K_\eps(1+\|V\|_{L^\infty}^{2/3}),
\text{ and } h\leq h_\eps\left(\dfrac{T}{1+\|V\|_{L^\infty}^{2/3}}\right)^{1+\varepsilon}.
\end{cases}
\]

If $V\in \cc_b^1(\mathbb R^d)$, the same statement holds with $\|V\|_{L^\infty}^{2/3}$ replaced by $\|V\|_{W^{1,\infty}}^{1/2}$.
\end{theorem}

A similar observability result also holds when the potential satisfies a power growth condition. Inequality \eqref{eq:necessaryintro} can be interpreted as a quantitative unique continuation estimate for the discrete (reverse) heat equation. In contrast with the continuous setting, the presence of the remainder term $C e^{-C/h^2}\|v_F\|^2$ is unavoidable and reflects the failure of strong unique continuation for discrete parabolic equations.

The estimate also highlights a geometric constraint on the observation set. When the potential is bounded, observability can only hold if $\omega$ is thick, a condition closely related to equidistribution and consistent with known results in the continuous setting \cite{EV,WWZZ}. The proof relies in particular on heat kernel estimates on the lattice, which are of independent interest. For potentials with polynomial growth, the optimal geometric condition on $\omega$ remains partially open, and weaker assumptions are required.

\smallskip

\paragraph{\bf Organization of the paper.}
The paper is organized as follows. In Section~\ref{sec:properties} we introduce the notation and recall basic properties of discrete Schrödinger operators on the lattice $h\mathbb Z^d$, including the definition of spectral projectors and the classes of potentials under consideration.

Section~\ref{sec:proof-of-main-thms} is devoted to the proof of the spectral inequalities. We first establish a Carleman estimate adapted to the operator $P_h=-\Delta_h+V$ in the full-space discrete setting. This estimate extends the semi-discrete framework of Boyer--Hubert--Le~Rousseau to incorporate the presence of a potential. We then use this Carleman inequality to derive the spectral estimates for bounded and polynomial growth potentials.

In Section~\ref{sec:control} we apply these spectral inequalities through the Lebeau--Robbiano method to obtain
approximate null-controllability results and deduce
uniform observability estimates via the Hilbert Uniqueness Method.

Finally, Section~\ref{sec:necessary} discusses necessary geometric conditions on the observation set in the case of bounded potentials and establishes a partial converse showing that thickness is essentially required for observability.

\section{Setting and preliminaries on discrete Schrödinger operators}\label{sec:properties}

In this section, we give some basic notions and results about discrete Schrödinger operators for both bounded potentials and unbounded ones.

\subsection{Discrete calculus}

Let us first introduce the discrete operators that will be used throughout the paper. We take $u$ and $v$ to be functions on $h\Z^d$.
We will say that $u\in\ell^2(h\Z^d)$ if
$$
\|u\|_{\ell^2(h\Z^d)}^2:=\sum_{x\in h\Z^d}|u(x)|^2
$$
is finite.

Next, we define the discrete forward and backward difference operators as 
\begin{equation*}
    D_{j,h}^{\pm} u(x):=\pm \frac{\left(u(x\pm he_j)-u(x)\right)}{h},\quad  j=1,2,\cdots,d.
\end{equation*}
as well as
\begin{equation*}
    D_h^{\pm} u:=\bigl(D^{\pm}_{j,h} u\bigr)_{j=1,\dots,d}:=(D_{1,h}^\pm u, D_{2,h}^\pm u, \cdots,D_{d,h}^\pm u)
\end{equation*}
We also use $D_{h}:=(D_h^+,D_h^-)$ for simplicity.
We write
    $$
    \Delta_ju(x)=D_{j,h}^-D_{j,h}^+ u(x)=\frac{u(x+he_j)+u(x-he_j)-2u(x)}{h^2}.
    $$
and
$$
\Delta_hu(x)=\sum_{j=1}^d\Delta_ju(x).
$$
We will also need the unscaled Laplacian $h^2\Delta_h$ which is, up to sign, the graph Laplacian on $h\Z^d$:
    \begin{equation*}
        -h^2\Delta_h u(x)=\sum_{y\sim x} (u(x)-u(y)),
    \end{equation*}
    and the continuous Laplacian on $\R^d$
    \begin{equation*}
        \Delta u(x)=\sum_{j=1}^d \partial_j^2 u(x)= \sum_{j=1}^d u_{jj}(x).
    \end{equation*}

For each $j$, we define several operators:
    $$
    M_j^\pm u(x)=\frac{u(x\pm he_j)+u(x)}{2}, 
    $$
    and
    $$
    \dc_ju(x)=\frac{1}{2}(D_{j,h}^++D_{j,h}^-)u(x)=M_j^\pm D_{j,h}^\mp u(x)=\frac{u(x+he_j)-u(x-he_j)}{2h}.
    $$

We will need some discrete formulas that we recall for the convenience of the reader.
The product formula for the discrete derivative writes
    \begin{equation}\label{discrete:D:eq2}
        D_{j,h}^\pm(uv) = D_{j,h}^\pm u M_j^\pm v+M_j^\pm u D_{j,h}^\pm v.
    \end{equation}
    The discrete Laplacian of a product is
    \begin{equation}\label{discretelaplacefg:eq1} 
    \Delta_h (uv)=v\Delta_h u+u\Delta_h v+2\sum_{j=1}^dM_j^-\bigl(D_{j,h}^+uD_{j,h}^+v\bigr).
    \end{equation}

When $u$ and $v$ are supported in a cube $Q$ and $u=v=0$ on $\partial Q$,
the summation by parts formulas reduce to
    \begin{equation}\label{Summation:eq2}
        \sum_{x\in Q} v D_{j,h}^+ u = -\sum_{x\in Q} u D_{j,h}^- v .
    \end{equation}

\subsection{Potentials and control sets}

Throughout this paper, $V:\R^d\to\R^+$ is a real, non-negative valued function that will be restricted
to $h\Z^d$. As we let $h\to 0$, to avoid any useless technicality, we always assume that $V$ is continuous.
We will consider 3 classes of potentials:

-- bounded non-negative continuous potentials, $V\in \cc_b(\R^d,\R^+)$;

-- bounded non-negative potentials with bounded derivative, $V\in \cc_b^1(\R^d,\R^+)$. In this case, we
write $\|V\|_{W^{1,\infty}(\R^d)}=\|V\|_\infty+\|\nabla V\|_\infty$;

-- potentials with a power growth behavior. More precisely, we will consider potentials satisfying the following assumption:

\begin{assumption}{A}\label{assump}
    Let $V\in \cc(\R^d)$ be a nonnegative real valued potential that can be written as $V=V_1+V_2$ for some $V_1\in \cc^1(\R^d)$ and $V_2\in \cc(\R^d)$. Further, assume that there exist constants $c_1>0,c_2>0$ and $\beta_2\ge \beta_1>0$ such that 
    \begin{equation*}
        V(x)\ge c_1\langle x\rangle^{\beta_1}
    \end{equation*}
    and
    \begin{equation*}
        |DV_1(x)|+|V_1(x)|+|V_2(x)|^{\frac{4}{3}}\le c_2 |x|^{\beta_2}.
    \end{equation*}
\end{assumption}

Note that when $V(x)=|x|^\beta$ with $\beta>1$, we simply take $V_1=V$, $V_2=0$ and Assumption~\ref{assump}
is satisfied with $\beta_1=\beta_2=\beta$. When $0<\beta\leq 1$, we take $\chi$ a smooth cutoff function
syuch that $\chi(x)=1$ for $|x|\leq 1$ and $\chi(x)=0$ for $|x|\geq 2$. Then setting
$V_1=(1-\chi)V$, $V_2=\chi V$, we see that Assumption~\ref{assump}
is satisfied with $\beta_1=\beta$ and $\beta_2=\dfrac{4}{3}\beta$. On the other hand,
$V(x)=\scal{x}^\beta$, taking $V_1=V$ and $V_2=0$, we get that $V$ satisfies Assumption~\ref{assump}
with $\beta_1=\beta_2=\beta$.

We then consider the discrete Schr\"odinger operator
$$
P_hu=-\Delta_hu+Vu
$$
where $V$ is understood as the restriction of $V$ from $\R^d$ to $h\Z^d$ and the product $Vu$ is the pointwise
product $Vu(hk)=V(hk)u(hk)$ for $h>0$ and $k\in\Z^d$.

\smallskip

Our main results deal with spectral inequality, observability and controllability from
subsets of $h\Z^d$. However, our focus is on obtaining results that hold when $h\to 0$.
Therefore all sets will be taken in $\R^d$ and intersected with $h\Z^d$. In order to avoid heavy notation
we will use the same letter for both, the subset of $\R^d$ in its intersection with $h\Z^d$.

We can now describe the conditions on the set $\omega$ on which controllability properties hold.
First, $\omega\subset\R^d$. When $u$ is a function on $h\Z^d$, $\|u\|_{\ell^2(\omega)}$
means
$$
\|u\|_{\ell^2(\omega)}^2=\|u\|_{\ell^2(\omega\cap h\Z^d)}^2
=\sum_{\stackrel{k\in \Z^d}{hk\in\omega}}|u(hk)|^2.
$$
Further, we will always assume that $\omega$
is open with smooth boundary. This condition ensures that the geometry of
$\omega$ is well reflected in the sampling process when $h\to 0$. In particular,  for every ball $B$,
$$
h^d|\omega\cap h\Z^d\cap B|\to|\omega\cap B|
$$
when $h\to 0$. Here as usual,
$|E|$ means the cardinality for a subset $E$ of $h\Z^d$ while it denotes the Lebesgue measure of $E$ for a subset
of $\R^d$.

Further, we denote by $Q_{L}(x)=x+[-L/2,L/2]^d$ the cube centered at $x$ of side length $L$ in $\R^d$ and $Q_L:=Q_L(0)$.
When we work in a discrete space, we will still write $Q_L(x)$ for $Q_L(x)\cap h\Z^d$.
For convenience we always assume that $L/2$ is a multiple of $h$. We will consider the following notions of thickness of sets that are common in control theory:

\begin{definition}
    Let $\omega$ be an open subset of $\R^d$ with smooth boundary. Let $L>0$ and $\gamma\in(0,1]$.
    \begin{enumerate}
        \item We say that $\omega$ is $(L,\gamma)$-\emph{equidistributed} (abbreviated \emph{equidistributed}) if 
    there exists a sequence $\lbrace z_k \rbrace_{k\in\Z^d} \subset \R^d$ such that
    \begin{equation*}
        \omega\cap \bigl(L k+Q_L\bigr)\supset  Q_{\gamma L}(z_k).
    \end{equation*}  
    
        \item We say that $\omega$ is \emph{thick} (at scale $L>0$) if, for every $x\in\R^d$,
    $$
        \bigl|\omega\cap\bigl(x+Q_L\bigr)\bigl|\geq\gamma|Q_L|.
    $$ 
    \end{enumerate}
    \end{definition}
    
Equidistributed sets provide a simple and natural way to construct nontrivial thick sets, although other constructions exist, for instance through sublevel sets of multivariate polynomials. The two notions are closely related.
From a geometric viewpoint, equidistributed sets can be viewed as a regular subclass of thick sets.

\subsection{Schrödinger operators for bounded potentials}
For a real valued, non-negative bounded potential $V\in \cc_b(\R^d)$, we restrict $V$ to $h\Z^d$, and define the Schrödinger operator $P_h=-\Delta_h+V$, which is a bounded self-adjoint positive operator on $\ell^2(h\Z^d)$. Based on the spectral theorem (see for instance \cite[Section 2.5]{davi1995spec}), one can define the spectral projector by
\begin{equation}\label{def-project}
    \Pi_{\mu,h} u:=\un_{P_h\le \mu}u:=\int_{0}^\mu \d m_\lambda u
\end{equation}
where $\d m_\lambda$ is the spectral measure of $P_h$. Moreover, we have
\begin{equation*}
    F(P_h)=\int_{0}^\infty F(\lambda)\d m_\lambda,\quad \forall F\in L^\infty(\R),
\end{equation*}
which satisfies
\begin{equation*}
    \scal{ G(P_h)u, K(P_h)u }=\int_{0}^\infty G(\lambda)\overline{K(\lambda)} \langle \d m_\lambda u,u\rangle_{\ell^2(h\Z^d)},\quad \forall u\in \ell^2(h\Z^d)
\end{equation*}
for any $G,K\in L^\infty(\R)$. We define the spectral subspace $\mathcal{E}_\mu(P_h)$ for any $\mu\in\R$ as the space of functions with spectrum in $(-\infty,\mu]$, that is,
\begin{equation}
    \mathcal{E}_\mu(P_h):=\lbrace u\in \ell^2(h\Z^d):\Pi_{\mu,h} u=u \rbrace.\label{def-subspace}
\end{equation}

\subsection{Schrödinger operators for power growth potentials}
For a potential satisfying the confining condition $V(x)\to \infty$ when $|x|\to \infty$, there exists an orthonormal sequence in $\ell^2(h\Z^d)$ of eigenfunctions $\lbrace\phi_{k}\rbrace_{k\in\N}$ and the corresponding sequence of eigenvalues $\lbrace \lambda_{k}\rbrace_{k\in\N}$ such that 
\begin{equation}\label{a-4}
    (-\Delta_h+V) \phi_k=\lambda_k\phi_k,\quad \forall k\in\N,
\end{equation}
and 
\begin{equation*}
    \lambda_0\le \lambda_1\le \cdots\le \lambda_k\le\cdots\to \infty.
\end{equation*}
Here one should notice that each $\lambda_k$ and $\phi_k$ depend on $h$. For simplicity, we neglect the explicit dependence in those symbols. Under this setting, the spectral subspace is a finite-dimensional subspace of $\ell^2(h\Z^d)$ and it can be written in an explicit form
\begin{equation*}
    \mathcal{E}_\mu (P_h)=\bigl\lbrace \sum_{\lambda_k\le \mu}c_k\phi_k: c_k\in\C \bigr\rbrace.
\end{equation*}

The following proposition is a discrete analogue of the localization property.
\begin{proposition}\label{prp-localization}
    Let $V\ge c|x|^\beta$ for some positive constants $c>0$ and $\beta>0$. Then there exists a positive constant $\hat{C}:=\hat{C}(c,\beta)$, such that for any $\mu>0$ and any $\phi\in\mathcal{E}_\mu(P_h)$, we have
    \begin{equation}\label{a-3}
        \|\phi\|^2_{\ell^2(h\Z^d)}\le 2\|\phi\|^2_{\ell^2\left(Q_{\Lambda}(0)\right)}
    \end{equation}
    when $\Lambda\geq L_\mu ={\hat{C}\mu^{1/\beta}}$.
\end{proposition}

\begin{proof}
    Let $\phi\in \mathcal{E}_\mu(P_h)$ and we have
    \begin{equation*}
        \begin{aligned}
            \| |x|^{\frac{\beta}{2}} \phi\|^2_{\ell^2(h\Z^d)}=\frac{1}{c}\langle c |x|^\beta \phi,\phi \rangle_{\ell^2(h\Z^d)} \le \frac{1}{c}\langle P_h\phi,\phi\rangle \le \frac{\mu}{c} \|\phi\|^2_{\ell^2(h\Z^d)},
        \end{aligned}
    \end{equation*}
    where the last inequality follows by functional calculus and $\phi\in \mathcal{E}_\mu(P_h)$. Therefore one has
    \begin{eqnarray*}
     \|\phi\|^2_{\ell^2(h\Z^d\backslash Q_{L_\mu})}
     &=&\| |x|^{-\frac{\beta}{2}}|x|^{\frac{\beta}{2}} \phi\|^2_{\ell^2\left(h\Z^d\backslash Q_{L_\mu}\right)}
     \le (L_\mu/2)^{-\beta}\||x|^{\frac{\beta}{2}} \phi\|^2_{\ell^2\left(h\Z^d\backslash Q_{L_\mu}\right)}\\
     &\leq&(L_\mu/2)^{-\beta}\||x|^{\frac{\beta}{2}} \phi\|^2_{\ell^2(h\Z^d)}
     \leq \frac{ 2^\beta\mu }{c L_\mu^\beta}\|\phi\|^2_{\ell^2(h\Z^d)}\le \frac{1}{2}\|\phi\|^2_{\ell^2(h\Z^d)},
    \end{eqnarray*}
    provided the value of $L_\mu$ is chosen as
    \begin{equation*}
        L_\mu= 2\left(\frac{2\mu}{c}\right)^{\frac{1}{\beta}}.
    \end{equation*}
    This is equivalent to \eqref{a-3} so we completed the proof.
\end{proof}

We also have the following local Caccioppoli inequality:
\begin{proposition}\label{prp-local-Cacci}
    Let $h<\dfrac{1}{2}$, $L>1$, $x_0 \in h\Z^d$ and $V\in \cc_b^\infty\bigl(Q_{2L}(x_0)\bigr)$. Let $\phi$ satisfy $(-\Delta_h+V)\phi=0$ in $Q_{2L}(x_0)$. We have
    \begin{equation*}
        \| D_h^\pm\phi \|^2_{\ell^2(Q_L(x_0))}\le 2\left(\frac{264d}{L^2}+\|V\|_{L^\infty(Q_{2L})}\right)\|\phi\|^2_{\ell^2(Q_{2L}(x_0))}.
    \end{equation*}
\end{proposition}

\begin{proof}
We will only prove the inequality for $D_h^+$, the argument for $D_h^-$ is similar.
Without loss of generality, we just fix $x_0=0$. Choose a cutoff function $\chi\in C_c^\infty (Q_{3L/2})$ such that $\chi=1$ in $Q_{L}$ and, for $j=1,\ldots,d$, $|D_{j,h}^\pm\chi|<\frac{8}{L}$. Then
\begin{equation}
\| D_h^+\phi \|^2_{\ell^2(Q_L)}\leq \| \chi D_h^+\phi \|^2_{\ell^2(Q_{2L})}
=\sum_{j=1}^d\langle \chi D_{j,h}^+\phi,\chi D_{j,h}^+\phi\rangle_{\ell^2(Q_{2L})}.
\label{eq:prp-local-Cacci}
\end{equation}
But,
\begin{equation*}
        \langle \chi D_{j,h}^+\phi,\chi D_{j,h}^+\phi\rangle_{\ell^2(Q_{2L})}
        = \sum_{x\in Q_{2L}}\chi(x)^2 D_{j,h}^+\phi(x) D_{j,h}^+ \overline{\phi(x)}
        =-\sum_{x\in Q_{2L}} D_{j,h}^-\bigl(\chi^2 D_{j,h}^+\phi\bigr)(x)\overline{\phi(x)}
\end{equation*}
with the  integration by parts formula, taking into consideration that $\chi(x)=0$ outside $Q_{2L}$.
Next write
\begin{equation*}
\begin{aligned}
D_{j,h}^-\bigl(\chi^2 D_{j,h}^+\phi\bigr)(x)
=& \frac{1}{h}\bigl(\chi(x)^2D_{j,h}^+ \phi(x)-\chi(x-he_j)^2D_{j,h}^+\phi(x-he_j)\bigr)\\
=& \frac{1}{h}\Bigl(\chi(x)^2\bigl(D_{j,h}^+ \phi(x)-D_{j,h}^+\phi(x-he_j)\bigr)+\bigl(\chi(x)^2-\chi(x-he_j)^2\bigr)D_{j,h}^+\phi(x-he_j)\Bigr)\\
=& \chi(x)^2D_{j,h}^-D_{j,h}^+ \phi(x)+D_{j,h}^-\chi^2(x)D_{j,h}^+\phi(x-he_j)
\end{aligned}
\end{equation*}
Summing over all directions $j=1,\ldots,d$ and over $x\in Q_{2L}$, we thus obtain
\begin{equation}
        \| \chi D_h^+\phi \|^2_{\ell^2(Q_{2L})}
    =-\sum_{j=1}^d\sum_{x\in Q_{2L}}\chi(x)^2D_{j,h}^-D_{j,h}^+ \phi(x) \overline{\phi(x)}
    -\sum_{j=1}^d\sum_{x\in Q_{2L}}D_{j,h}^-\chi^2(x)D_{j,h}^+\phi(x-he_j)\overline{\phi(x)}.\label{a-7}
\end{equation}
We now bound each of these two sums. For the first sum, invert summations over $j$ and $x$ and use that
    \begin{equation*}
        \sum_{j=1}^d D_{j,h}^- D_{j,h}^+\phi (x) =\Delta_h \phi(x)=V(x)\phi(x).
    \end{equation*}
We then obtain
    \begin{align}
   \abs{\sum_{j=1}^d\sum_{x\in Q_{2L}}\chi(x)^2D_{j,h}^-D_{j,h}^+ \phi(x)\overline{\phi(x)}}
    &\leq\sum_{x\in Q_{2L}}\chi(x)^2\abs{\left(\sum_{j=1}^dD_{j,h}^-D_{j,h}^+ \phi(x)\right)\overline{\phi(x)}}\notag\\
    &=\sum_{x\in Q_{2L}}\chi(x)^2|V(x)|\,|\phi(x)|^2\leq \|V\|_\infty\|\phi\|^2_{\ell^2(Q_{2L})}\label{a-66}
   \end{align}
 since $\chi\leq 1$.

For the second term in \eqref{a-7}, using Young's inequality we get 
\begin{align}
\Bigl|\sum_{j=1}^d&\sum_{x\in Q_{2L}}D_{j,h}^-\chi^2(x)D_{j,h}^+\phi(x-he_j)\overline{\phi(x)}\Bigr| \notag \\
\leq& \sum_{j=1}^d\sum_{x\in Q_{2L}}2\frac{L}{32}\abs{D_{j,h}^-\chi^2(x)D_{j,h}^+\phi(x-he_j)}\,\frac{16}{L}|\phi(x)|\notag\\
\leq&\sum_{j=1}^d\left(\left(\frac{L}{32}\right)^2\norm{D_{j,h}^-\chi^2D_{j,h}^+\phi(\cdot-he_j)}_{\ell^2(Q_{2L})}^2+\left(\frac{16}{L}\right)^2\|\phi\|_{\ell^2(Q_{2L})}^2\right).\label{a-6}
\end{align}
Now write
\begin{eqnarray*}
D_{j,h}^-\chi^2(x)&=&\chi(x)D_{j,h}^-\chi(x)+\chi(x-he_j)D_{j,h}^-\chi(x)\\
&=&\bigl(\chi(x)-\chi(x-he_j)\bigr)D_{j,h}^-\chi(x)+2\chi(x-he_j)D_{j,h}^-\chi(x)
\end{eqnarray*}
so that
$$
D_{j,h}^-\chi^2(x)D_{j,h}^+\phi(x-he_j)=\bigl(D_{j,h}^-\chi(x)\bigr)^2\bigl(\phi(x)-\phi(x-he_j)\bigr)
+2\chi(x-he_j)D_{j,h}^-\chi(x)D_{j,h}^+\phi(x-he_j).
$$
Using the bound $|D_{j,h}^-\chi(x)|\leq\dfrac{8}{L}$ and the support properties of $\chi$, we obtain
$$
\norm{D_{j,h}^-\chi^2D_{j,h}^+\phi(\cdot-he_j)}_{\ell^2(Q_{2L})}\leq \left(\frac{8}{L}\right)^2\|\phi\|_{\ell^2(Q_{2L})}
+\frac{16}{L} \|\chi D_{j,h}^+\phi\|_{\ell^2(Q_{2L})}
$$
and therefore
$$
\norm{D_{j,h}^-\chi^2D_{j,h}^+\phi(\cdot-he_j)}_{\ell^2(Q_{2L})}^2\leq 2\left(\frac{8}{L}\right)^4\|\phi\|_{\ell^2(Q_{2L})}^2
+2\left(\frac{16}{L}\right)^2 \|\chi D_{j,h}^+\phi\|_{\ell^2(Q_{2L})}^2.
$$
Thus
$$
\left(\frac{L}{32}\right)^2\norm{D_{j,h}^-\chi^2D_{j,h}^+\phi(\cdot-he_j)}_{\ell^2(Q_{2L})}^2\leq
\frac{8}{L^2}\|\phi\|_{\ell^2(Q_{2L})}^2
+\frac{1}{2}\|\chi D_{j,h}^+\phi\|_{\ell^2(Q_{2L})}^2.
$$
Substituting this into \eqref{a-6} we obtain
\begin{align}
\Bigl|\sum_{j=1}^d&\sum_{x\in Q_{2L}}D_{j,h}^-\chi^2(x)D_{j,h}^+\phi(x-he_j)\overline{\phi(x)}\Bigr| \notag \\
\leq& \sum_{j=1}^d\left(\frac{264}{L^2}\|\phi\|_{\ell^2(Q_{2L})}^2
+\frac{1}{2}\|\chi D_{j,h}^+\phi\|_{\ell^2(Q_{2L})}^2\right)
\end{align}
Injecting the result together with \eqref{a-66} into \eqref{a-7} we obtain
    $$
    \sum_{j=1}^d\|\chi(x)D_{j,h}^+\phi\|_{\ell^2(Q_{2L})}^2\leq \frac{1}{2}\sum_{j=1}^d\|\chi(x)D_{j,h}^+\phi\|_{\ell^2(Q_{2L})}^2+
    \left(\|V\|_\infty+\frac{264d}{L^2}\right)\|\phi\|_{\ell^2(Q_{2L})}^2
    $$
    Absorbing the first term on the left in the right-hand side and using \eqref{eq:prp-local-Cacci}, we
    obtain the desired inequality.
\end{proof}

\section{Proof of main theorems}\label{sec:proof-of-main-thms}

\subsection{Carleman estimates}

In this section, we give a Carleman estimate associated to the semi-discrete equation defined by $P_h=-\Delta_h+V$,
\begin{equation*}
\begin{cases}
    (\partial_t^2-P_h)u(t,x)=0, \\
    u=0 \quad\mbox{on}\quad \{0\}\times Q, \\
    u=0 \quad \mbox{on} \quad (0,T_*)\times \partial Q.
\end{cases}
\end{equation*}
Here $Q$ is an arbitrary cube of side length $L$, $x+[-L/2,L/2]^d\subset h\Z^d$.

Before stating the Carleman estimate, we take a valid weight function $\varphi$ defined on a neighborhood
of $\overline{(0,T_*)\times Q_{L}}$ in $\R^d$. We first take $\mathcal{V}$ to be a bounded open set with smooth boundary containing $Q_L$.
We then take a smooth function $\psi$ defined on $\mathcal{V}$ that satisfies
the following properties
\begin{eqnarray*}
    |\nabla\psi|\geq c \mbox{ and } \psi>0 \mbox{ in } (0,T^*)\times Q_{L},  \quad && \\
     \partial_{n}\psi(t,x)<0 \mbox{ in } (0,T^*)\times \partial \mathcal{V}, \quad &&\partial_{i}^2\psi(t,x)\geq0 \mbox{ in } (0,T^*)\times \partial \mathcal{V}, \\
     \partial_t \psi \geq c \mbox{ on } \{0\}\times(Q_L\setminus\omega), \quad && \psi=\mbox{Cst and } \partial_t \psi \leq -c \mbox{ on }  \{T^*\}\times Q_L.
\end{eqnarray*}
Here $\partial_{n}$ is the outward unit normal to $\mathcal{V}$.
These conditions on $\psi$ are the conditions found in \cite[Assumption 1.3]{boye2010discr2}.
When $\mathcal{V}$ is a sufficiently small neighborhood of $Q_L$, the construction of such a $\psi$ is described in \cite[Appendix~A]{boye2010discr2}.
Finally, we take $\lambda$ a parameter and we set $\ffi=e^{\lambda\psi}$.
Note that $\ffi$ is smooth over $Q$ so that $\Delta\ffi$ is bounded.

We extend the Carleman estimate~\cite[Theorem 1.4]{boye2010discr2} (see also \cite[Theorem 2.2]{boye2010disc} for a 1-dimensional version) from $\partial_t^2+\Delta_h$ to the operator $P_h$:
\begin{theorem}\label{crc:Carleman:1}
Let $T_*>0$, $Q$ a cube and $\qc=[0,T_*]\times Q$.
    Let $\ffi$ be a weight function satisfying the above conditions.
    Let $V_1\in \cc_b^1(Q)$, $V_2\in \cc_b(Q)$ and $V=V_1+V_2$.

    There exists $s_0>1,h_0>0,\eps_0>0$ such that, when $h\leq h_0$, $sh\leq\eps_0$ and 
    $s>s_0\bigl(\|V_1\|_{W^{1,\infty}(Q)}^{\frac{1}{2}}+\|V_2\|_{L^\infty(Q)}^{\frac{2}{3}}+1\bigr)$,
    then, for every $u\in \cc^2([0,T_*]\times Q,\mathbb{C})$ satisfying $u=0$ on $\lbrace 0\rbrace \times Q$ and $u=0$ on $(0,T_*)\times \partial Q$, we have
    \begin{multline}\label{eq:Carleman:3.1}
        s^3\|e^{s\ffi}u\|^2_{L^2(\qc)}+s\|e^{s\ffi}\partial_t u\|^2_{L^2(\qc)}+
        s\|e^{s\ffi}D^+_h u\|^2_{L^2(\qc)} 
        +s\|e^{s\ffi}D^-_h u\|^2_{L^2(\qc)} \\
        +s\|e^{s\ffi(0,\cdot)}\partial_tu(0,\cdot)\|^2_{\ell^2(Q)}+se^{2s\ffi(T_*)}\|\partial_tu(T_*,\cdot)\|^2_{\ell^2(Q)}
        +s^3e^{2s\ffi(T_*)}\|u(T_*,\cdot)\|^2_{\ell^2(Q)} \\
        \lesssim \|e^{s\ffi}(\partial_t^2-P_h)u\|^2_{L^2(\qc)}+se^{2s\ffi(T_*)}\|D_h^+ u(T_*,\cdot)\|^2_{\ell^2(Q)} \\
        +se^{2s\ffi(T_*)}\|D_h^- u(T_*,\cdot)\|^2_{\ell^2(Q)} 
        +s\|e^{s\ffi(0,\cdot)}\partial_t u(0,\cdot)\|^2_{\ell^2(\omega)}.
    \end{multline}
\end{theorem}

\begin{proof}
    Without loss of generality, we can assume that $u$ is real.
    Let $\ffi$ be defined as above, and let $v=e^{s\ffi}u$.
    We have
    \begin{equation}\label{Carleman:1}
            e^{s\ffi}(\partial^2_t-P_h)u=e^{s\ffi}(\partial_t^2-P_h)(e^{-s\ffi}v)
             =e^{s\ffi}(\partial_t^2e^{-s\ffi}v) + e^{s\ffi}\sum_{j=1}^dD_{j,h}^+D_{j,h}^-(e^{-s\ffi}v)-Vv.
    \end{equation}
   First we see that 
   $$
   e^{s\ffi}\partial_t^2(e^{-s\ffi}v)=e^{s\ffi}(\partial_t^2e^{-s\ffi})v+2e^{s\ffi}(\partial_te^{-s\ffi})\partial_tv+\partial_t^2v.
   $$
Further, applying \eqref{discrete:D:eq2} twice, we have
   \begin{equation*}
       \begin{aligned}
           e^{s\ffi}D_{j,h}^+&D_{j,h}^-(e^{-s\ffi}v)  \\
           =& e^{s\ffi}(M_j^+M_j^-e^{-s\ffi})D_{j,h}^+D_{j,h}^-v+e^{s\ffi}(D_{j,h}^+D_{j,h}^-e^{-s\ffi})M_j^+M_j^-v +2e^{s\ffi}\dc_{j}e^{-s\ffi}\dc_{j}v.
       \end{aligned}
   \end{equation*}
We now introduce 
    \begin{equation*}
    \begin{aligned}
        \mathcal{A}_1v&=\partial_t^2v+\sum_{j=1}^d e^{s\ffi}(M_j^+M_j^-e^{-s\ffi})D_{j,h}^+D_{j,h}^-v, \\
        \mathcal{A}_2v&=e^{s\ffi}(\partial_t^2e^{-s\ffi})v+\sum_{j=1}^de^{s\ffi}(D_{j,h}^+D_{j,h}^-e^{-s\ffi})M_j^+M_j^-v, \\
        \mathcal{B}_1v&=2e^{s\ffi}(\partial_te^{-s\ffi})\partial_tv+\sum_{j=1}^d2e^{s\ffi}\dc_je^{-s\ffi}\dc_jv,\\
        \mathcal{B}_2v&=-2s(\Delta\ffi)v,
    \end{aligned}
    \end{equation*}
and $\mathcal{A}v=\mathcal{A}_1v+\mathcal{A}_2v-V_1v$ as well as $\mathcal{B}v=\mathcal{B}_1v+\mathcal{B}_2v$.
Equation \eqref{Carleman:1} then reads 
    $$
    e^{s\ffi}(\partial_t^2-P_h)(e^{-s\ffi}v)=\mathcal{A}v+\mathcal{B}v-V_2v+2s(\Delta \ffi)v.
    $$
    As $|\Delta\ffi|\leq C_\ffi$, where $C_\ffi$ is a constant independent of $s$, we obtain
    $$
    \|\mathcal{A}v+\mathcal{B}v\|^2_{L^2(\qc)}\lesssim\|e^{s\ffi}(\partial_t^2-P_h)u\|^2_{L^2(\qc)}+\|V_2v\|^2_{L^2(\qc)}+s^2\|v\|^2_{L^2(\qc)}.
    $$
    On the other hand
    \begin{eqnarray*}
\|\mathcal{A}v+\mathcal{B}v\|^2_{L^2(\qc)}
&=&\|\mathcal{A}v\|^2_{L^2(\qc)}+\|\mathcal{B}v\|^2_{L^2(\qc)}+2\Re\scal{\mathcal{A}v,\mathcal{B}v}_{L^2(\qc)}\notag\\
&\geq&2\Re\scal{\mathcal{A}v,\mathcal{B}v}_{L^2(\qc)}\notag\\
&=&2\sum_{j,k=1}^2\Re\scal{\mathcal{A}_jv,\mathcal{B}_kv}_{L^2(\qc)}
-2\sum_{k=1}^2\Re\scal{V_1v,\mathcal{B}_kv}_{L^2(\qc)}.
    \end{eqnarray*}
 We thus obtain
\begin{multline}
\label{eq:extra3}
\sum_{j,k=1}^2\Re\scal{\mathcal{A}_jv,\mathcal{B}_kv}_{L^2(\qc)}\\
\lesssim\|e^{s\ffi}(\partial_t^2-P_h)u\|^2_{L^2(\qc)}+\|V_2v\|^2_{L^2(\qc)}+s^2\|v\|^2_{L^2(\qc)}
+\sum_{k=1}^2|\scal{V_1v,\mathcal{B}_kv}_{L^2(\qc)}|.
\end{multline}

The four scalar products $\scal{\mathcal{A}_jv,\mathcal{B}_kv}_{L^2(\qc)}$ on the left have already been lower bounded in \cite[Section 3]{boye2010disc}, who
established that
$$
    \begin{aligned}
\sum_{j,k=1}^2\Re\scal{\mathcal{A}_jv,\mathcal{B}_kv}_{L^2(\qc)}\geq& s^3\|e^{s\ffi}u\|^2_{L^2(\qc)}+s\|e^{s\ffi}\partial_t u\|^2_{L^2(\qc)}
        + s\|e^{s\ffi}D_h^+ u\|^2_{L^2(\qc)}+ s\|e^{s\ffi}D_h^- u\|^2_{L^2(\qc)} \\
        &+s\|e^{s\ffi(0,\cdot)}\partial_tu(0,\cdot)\|^2_{\ell^2(Q)}+se^{2s\ffi(T_*)}\|\partial_tu(T_*,\cdot)\|^2_{\ell^2(Q)}\\
        &+s^3e^{2s\ffi(T_*)}\|u(T_*,\cdot)\|^2_{\ell^2(Q)} -se^{2s\ffi(T_*)}\|D_h^+ u(T_*,\cdot)\|^2_{\ell^2(Q)} \\
        &-se^{2s\ffi(T_*)}\|D_h^- u(T_*,\cdot)\|^2_{\ell^2(Q)} 
        -s\|e^{s\ffi(0,\cdot)}\partial_t u(0,\cdot)\|^2_{\ell^2(\omega)}.
        \end{aligned}
$$
It therefore remains to estimate the two scalar products $\scal{V_1v,\mathcal{B}_1v}_{L^2(\qc)}$
and $\scal{V_1v,\mathcal{B}_2v}_{L^2(\qc)}$.

For the first scalar product, we write
\begin{equation*}
    \begin{aligned}
         \scal{V_1v,\mathcal{B}_1v}=& 2\int_{(0,T_*)}\sum_{x\in Q}\sum_{j=1}^d V_1(x)v(t,x)e^{s\ffi(t,x)}\dc_je^{-s\ffi}(t,x)\dc_jv(t,x) \d t \\
         &+2\int_{(0,T_*)}\sum_{x\in Q} V_1(x)v(t,x)e^{s\ffi(t,x)}(\partial_te^{-s\ffi(t,x)})\partial_tv(t,x)\d t \\
         &= I_1+I_2        
    \end{aligned}
    \end{equation*}
    and we bound $I_1$ and $I_2$ separately. First
    \begin{equation*}
    \begin{aligned}
                I_1&=\bigg|\int_{(0,T_*)}\sum_{x\in Q}\sum_{j=1}^d \frac{1}{2h} V_1(x)e^{s\ffi(t,x)}\dc_je^{-s\ffi}(t,x) v(t,x)\left(v(t,x+he_j)-v(t,x-he_j)\right) \d t \bigg|\\
                &=\bigg|\int_{(0,T_*)}\sum_{x\in Q}\sum_{j=1}^d \frac{1}{h}\bigg(V_1(x)e^{s\ffi(t,x)}\dc_je^{-s\ffi}(t,x)-V_1(x+he_j)e^{s\ffi(t,x+he_j)}\dc_je^{-s\ffi(t,x+he_j)}\bigg)\\
                &\qquad\qquad\qquad\times v(t,x)v(t,x+he_j) \d t\bigg| \\
                &=\bigg|\int_{(0,T_*)}\sum_{x\in Q}\sum_{j=1}^dD_{j,h}^+\bigl[V_1e^{s\ffi}\dc_je^{-s\ffi}\bigr](x)v(t,x)v(t,x+he_j) \d t\bigg| \\
                &\lesssim \norm{\sum_{j=1}^dD_{j,h}^+\bigl[V_1e^{s\ffi}\dc_je^{-s\ffi}\bigr]}_\infty\|v\|_{L^2(\qc)}^2\leq
                \left(\sum_{j=1}^d\norm{D_{j,h}^+\bigl[V_1e^{s\ffi}\dc_je^{-s\ffi}\bigr]}_\infty\right)\|v\|_{L^2(\qc)}^2
    \end{aligned}
    \end{equation*}
    by Cauchy-Schwarz. We now estimate the $L^\infty$ norm in this expression.
    The variable $t$ plays no role and, to simplify notation, we omit it.
    Write
 \begin{equation*}
        G_j(x):=\frac{1}{h}V_1(x)\Psi_j(x), \text{ and }
        \Psi_j(x):=\frac{1}{2}\left(e^{s(\ffi(x)-\ffi(x+he_j))}-e^{s(\ffi(x)-\ffi(x-he_j))}\right),
    \end{equation*}
so that $D_{j,h}^+G_j(x)=h^{-1}D_{j,h}^+V_1(x)\Psi_j(x+he_j)+h^{-1}V_1(x)D_{j,h}^+\Psi_j(x)$ thus 
\begin{eqnarray*}
\norm{D_{j,h}^+\bigl[V_1e^{s\ffi}\dc_je^{-s\ffi}\bigr]}_\infty\leq
h^{-1}\|\partial_jV_1\|_\infty \|\Psi_j\|_\infty+h^{-1}\|V_1\|_\infty\|D_{j,h}^+\Psi_j(x)\|_\infty.
\end{eqnarray*}
Now $|s(\ffi(x)-\ffi(x\pm he_j))|\lesssim sh\lesssim 1$ since $\partial_j\ffi$ is bounded and $sh\leq\eps_0$.
Further 
$$
|s(\ffi(x)-\ffi(x+he_j))-s(\ffi(x)-\ffi(x-he_j))|=s|\ffi(x+he_j)-\ffi(x-he_j)|\lesssim sh.
$$
From the Mean Value Theorem, $|e^a-e^b|\lesssim |a-b|$ when $|a|,|b|\lesssim 1$,
we obtain
$$
|\Psi_j(x)|=\frac{1}{2}\abs{e^{s(\ffi(x)-\ffi(x+he_j))}-e^{s(\ffi(x)-\ffi(x-he_j))}}\lesssim sh.
$$
As $\|D_{j,h}^+\Psi_j(x)\|_\infty\leq\|\partial_j\Psi_j\|_\infty$ and
\begin{equation*}
\begin{aligned}
\partial_j\Psi_j(x)=&\frac{s}{2}\Bigl(e^{s(\ffi(x)-\ffi(x+he_j))}\bigl(\partial_j\ffi(x)-\partial_j\ffi(x+he_j)\bigr)
-e^{s(\ffi(x)-\ffi(x-he_j))}\bigl(\partial_j\ffi(x)-\partial_j\ffi(x-he_j)\bigr)\Bigr)\\
=&\frac{s}{2}\bigl(e^{s(\ffi(x)-\ffi(x+he_j))}-e^{s(\ffi(x)-\ffi(x-he_j))}\bigr)\bigl(\partial_j\ffi(x)-\partial_j\ffi(x+he_j)\bigr)\\
&+\frac{s}{2}e^{s(\ffi(x)-\ffi(x-he_j))}\bigl(\partial_j\ffi(x-he_j)-\partial_j\ffi(x+he_j)\bigr).
\end{aligned}
\end{equation*}
A similar argument shows that $\|\partial_j\Psi_j\|\lesssim (sh)^2+sh\lesssim sh$ since $sh\lesssim 1$. This leads to
$$
\norm{D_{j,h}^+\bigl[V_1e^{s\ffi}\dc_je^{-s\ffi}\bigr]}_\infty\lesssim s\|V_1\|_{W^{1,\infty}(Q)}
$$
and then
    \begin{equation}\label{eq:estV1B1:eq1}
        I_1\lesssim s\|V_1\|_{W^{1,\infty}(Q)}\|v\|_{L^2(\qc)}^2.
    \end{equation}

    For the second integral 
    \begin{equation}\label{eq:estV1B1:eq2}
        \begin{aligned}
            |I_2|&\leq s\sum_{x\in Q}|V_1(x)|\left|\int_{(0,T_*)} \partial_t\ffi(t,x)\partial_tv^2(t,x)\d t\right| \\
            &\leq s \sum_{x\in Q}|V_1(x)|\abs{\left[\partial_t\ffi(t,x)v^2(t,x)\right]_0^{T_*}} + s\sum_{x\in Q}|V_1(x)|\int_{(0,T_*)} |\partial_t^2\ffi(t,x)||v(t,x)|^2\d t \\
            &\leq s \|V_1\|_{L^{\infty}(Q)}(\|v\|^2_{L^2(\qc)}+\|v(T_*,\cdot)\|^2_{\ell^2(Q)})\\
            &\leq s \|V_1\|_{W^{1,\infty}(Q)}\bigl(\|v\|^2_{L^2(\qc)}+\|v(T_*,\cdot)\|^2_{\ell^2(Q)}\bigr).
        \end{aligned}
    \end{equation}
    where we integrated by parts in the second line and used that $\partial_t\ffi,\partial_t^2\ffi$ are bounded in the last line.

The second one is the easier:
    \begin{equation}
    \begin{aligned}
        |\scal{V_1v,B_2v}|&=2\abs{\int_{(0,T_*)}\sum_{x\in Q} s\Delta\ffi(t,x)V_1(x)|v(t,x)|^2 \d t} \\
        &\lesssim s \|V_1\|_{L^{\infty}(Q)}\|v\|^2_{L^2(\qc)}\leq  
        s \|V_1\|_{W^{1,\infty}(Q)}\|v\|^2_{L^2(\qc)}
    \end{aligned} 
    \label{eq:estV1B2}
    \end{equation}
since $|\Delta\ffi|\lesssim 1$.

Grouping the three estimates \eqref{eq:estV1B1:eq1}, \eqref{eq:estV1B1:eq2} and \eqref{eq:estV1B2},  we now get
\begin{align*}
\|V_2v\|^2_{L^2(\qc)}&+\sum_{k=1}^2|\scal{V_1v,\mathcal{B}_kv}_{L^2(\qc)}|\\
\lesssim& 
\|V_2\|_{L^\infty(Q)}^2\|v\|^2_{L^2(\qc)}+
    s\|V_1\|_{W^{1,\infty}(Q)}\|v\|_{L^2(\qc)}^2+\|v(T_*,\cdot)\|^2_{\ell^2(Q)}\\
     =&\bigl(\|V_2\|_{L^\infty(Q)}^2 +s\|V_1\|_{W^{1,\infty}(Q)}\bigr)\|e^{s\ffi}u\|^2_{L^2(\qc)}+
     s\|V_1\|_{W^{1,\infty}(Q)}e^{s\ffi(T_*)}\|u(T_*,\cdot)\|^2_{\ell^2(Q)}.
\end{align*}

We have now estimated each of the terms on the right-hand side of \eqref{eq:extra3} that
was not present in \cite{boye2010discr2}. Adding the lower bounds of the left hand side from
\cite{boye2010discr2}, we get
    \begin{equation*}
    \begin{aligned}
       s^3\|e^{s\ffi}u&\|^2_{L^2(\qc)}
       +s\|e^{s\ffi}\partial_t u\|^2_{L^2(\qc)}
       + s\|e^{s\ffi}D_h^+ u\|^2_{L^2(\qc)}
       + s\|e^{s\ffi}D_h^- u\|^2_{L^2(\qc)}\\
        +s\|&e^{s\ffi(0,\cdot)}\partial_tu(0,\cdot)\|^2_{\ell^2(Q)}
        +se^{2s\ffi(T_*)}\|\partial_tu(T_*,\cdot)\|^2_{\ell^2(Q)}
        +s^3e^{2s\ffi(T_*)}\|u(T_*,\cdot)\|^2_{\ell^2(Q)} \\
        \lesssim& \|e^{s\ffi}(\partial_t^2-P_h)u\|^2_{L^2(\qc)}
        +se^{2s\ffi(T_*)}\|D_h^+ u(T_*,\cdot)\|^2_{\ell^2(Q)}
        +se^{2s\ffi(T_*)}\|D_h^- u(T_*,\cdot)\|^2_{\ell^2(Q)} \\
        &+s\|e^{s\ffi(0,\cdot)}\partial_t u(0,\cdot)\|^2_{\ell^2(\omega)} 
        +\bigl(\|V_2\|_{L^\infty(Q)}^2+s \|V_1\|_{W^{1,\infty}(Q)}+s^2\bigr)\|e^{s\ffi}u\|^2_{L^2(\qc)}\\
        &+s \|V_1\|_{W^{1,\infty}(Q)}e^{s\ffi(T_*)}\|u(T_*,\cdot)\|^2_{\ell^2(Q)}.
    \end{aligned}
    \end{equation*}
    Finally, taking $s\gtrsim  \|V_1\|_{W^{1,\infty}(Q)}^{\frac{1}{2}}+\|V_2\|_{L^\infty(Q)}^{\frac{2}{3}}+1$
    allows us to absorb 
    \begin{equation*}
       (s\|V_1\|_{W^{1,\infty}(Q)}+\|V_2\|^2_{L^\infty(Q)}+s^2)\|e^{s\ffi}u\|^2_{L^2(\qc)}
    \end{equation*}   
       by the term $ s^3\|e^{s\ffi}u\|^2_{L^2(\mathcal{Q})}$
       from the left hand side, and also to absorb
    $$
    s\|V_1\|_{W^{1,\infty}(Q)}e^{2s\ffi(T_*)}\|u(T_*,\cdot)\|^2_{\ell^2(Q)}
    $$
        by $ s^3e^{2s\ffi(T_*)}\|u(T_*,\cdot)\|^2_{\ell^2(Q)}$. This completes the proof.
\end{proof}

\subsection{Proof of the spectral inequalities for bounded potentials}
In this section, we prove the spectral inequalities when the potential is bounded and non-negative.

\begin{proof}[Proof of Theorem~\ref{th:main:1}]
Let $v\in \mathcal{E}_\mu(P_h)$ and define
\begin{equation*}
    \widetilde{v}(t,x):=\int_{0}^\mu \frac{\sinh (\sqrt{\lambda}t)       }{\sqrt{\lambda}} \d m_\lambda v.
\end{equation*}
Then one can easily verify that 
\begin{equation}
    (\partial_t^2-P_h) \widetilde{v}=0 
\end{equation}
and 
\begin{equation*}
    \partial_t \widetilde{v}\lvert_{t=0}=v.
\end{equation*}

Let $\omega$ be an equidistributed subset of $\R^d$ and define $\omega_k=\omega\cap (Lk+Q_L)$ for any $k\in\Z^d$.
Each $\omega_k$ contains a cube $\gamma Q_{L}(x_k)$ and, without loss of generality,
we may shrink $\omega_k$ so that $\omega_k=\gamma Q_{L}(x_k)$. Note that $\omega_k\subset Lk+Q_{3L}$.

Let $0\leq \chi\leq 1$ be a smooth cutoff function satisfying
\begin{equation}
    \chi =1 \text{ in } Q_{3L},\,\, \chi =0 \text{ in } \R^{d}\backslash {Q_{6L}}. \label{chi:q}
\end{equation}
Fix $k$ in $\Z^d$ and let $\chi_k(x) = \chi(x-Lk)$ and define a new function 
$$
u(t,x)=\chi_k (x) \widetilde{v}(t,x).
$$
A simple computation shows that 
\begin{equation*}
    D_{j,h}^\pm u= \chi_k D_{j,h}^\pm \widetilde{v} + \widetilde{v} D_{j,h}^\pm \chi_k \pm hD_{j,h}^\pm \chi_k D_{j,h}^\pm \widetilde{v}
\end{equation*}
From \eqref{discretelaplacefg:eq1} 
\begin{eqnarray*}
(\partial_t^2-P_h)u&=&(\partial_t^2-P_h)\chi_k\widetilde{v}\\
&=&\chi_k(\partial_t^2-P_h)\widetilde{v}-\tilde v \Delta_h\chi_k-2\sum_{j=1}^d M_j^-(D_{j,h}^+\chi_kD_{j,h}^+\widetilde{v}) \\
&=&-\widetilde{v}\Delta_h\chi_k-\sum_{j=1}^d(D_{j,h}^+\chi_kD_{j,h}^+\widetilde{v}+D_{j,h}^-\chi_kD_{j,h}^-\widetilde{v}) .
\end{eqnarray*}
We can now apply the Carleman estimate in Theorem~\ref{crc:Carleman:1} to $u$ in $Q_k=kL+Q_{6L}$. (thus $\qc=\qc_k=[0,3]\times Q_k$) with $T_*=3$: if
\begin{equation}\label{splusgrand:eq1}
s\geq s_0(1+\|V\|_\infty^{\frac{2}{3}})
\end{equation}
for $s_0$ sufficiently large, and using the triangular inequality together with the fact that $sh\lesssim 1$, we obtain
\begin{equation*}
      \begin{aligned}
        s^3\|e^{s\ffi}\chi_k  \widetilde{v}\|^2_{L^2(\qc_k)}&
        +s\|e^{s\ffi}\chi_k \partial_t \widetilde{v}\|^2_{L^2(\qc_k)} \\
        +s\|e^{s\ffi}\chi_k &D_h^+\widetilde{v}\|^2_{L^2(\qc_k)}
        -2s\|e^{s\ffi}\widetilde{v}D_h^+\chi_k \|^2_{L^2(\qc_k)}
        -2\|e^{s\ffi}D_h^+\widetilde{v}D_h^+\chi_k \|^2_{L^2(\qc_k)} \\
        +s\|e^{s\ffi}\chi_k &D_h^-\widetilde{v}\|^2_{L^2(\qc_k)}
        -2s\|e^{s\ffi}\widetilde{v}D_h^-\chi_k \|^2_{L^2(\qc_k)}
        -2\|e^{s\ffi}D_h^-\widetilde{v}D_h^-\chi_k \|^2_{L^2(\qc_k)} \\
        +s\|e^{s\ffi(0,\cdot)}\chi_k &\partial_t\widetilde{v}(0,\cdot)\|^2_{\ell^2(Q_k)}
        +se^{2s\ffi(3)}\|\chi_k \partial_t\widetilde{v}(3,\cdot)\|^2_{\ell^2(Q_k)}
        +s^3e^{2s\ffi(3)}\|\chi_k \widetilde{v}(3,\cdot)\|^2_{\ell^2(Q_k)} \\
        \lesssim&
        \|e^{s\ffi}D_h^+\chi_k D_h^+\widetilde{v}\|^2_{L^2(\qc_k)}
         +\|e^{s\ffi}D_h^-\chi_k D_h^-\widetilde{v}\|^2_{L^2(\qc_k)}
        +\|e^{s\ffi}\widetilde{v}\Delta_h\chi_k \|^2_{L^2(\qc_k)} \\
        &+se^{2s\ffi(3)}\|(\chi_k+hD_h^+\chi_k)D_h^+\widetilde{v}(3,\cdot)\|^2_{\ell^2(Q_k)} 
        +se^{2s\ffi(3)}\|\widetilde{v}(3,\cdot)D_h^+\chi_k \|^2_{\ell^2(Q_k)} \\
        &+se^{2s\ffi(3)}\|(\chi_k-hD_h^-\chi_k)D_h^-\widetilde{v}(3,\cdot)\|^2_{\ell^2(Q_k)} 
        +se^{2s\ffi(3)}\|\widetilde{v}(3,\cdot)D_h^-\chi_k \|^2_{\ell^2(Q_k)} \\
        &+s\|e^{s\ffi(0,\cdot)}\chi_k\partial_t \widetilde{v}(0,\cdot)\|^2_{\ell^2(\omega_k)}.
    \end{aligned}
    \end{equation*}
Now, we remove the positive terms
$$
s\|e^{s\ffi}\chi_k \partial_t \widetilde{v}\|^2_{L^2(\qc_k)},\ 
s\|e^{s\ffi(0,\cdot)}\chi_k \partial_t\widetilde{v}(0,\cdot)\|^2_{\ell^2(Q_k)}
\mbox{ and }
se^{2s\ffi(3)}\|\chi_k \partial_t\widetilde{v}(3,\cdot)\|^2_{\ell^2(Q_k)}
$$
from the left and rewrite the remaining ones as
\begin{equation}\label{a-11}
      \begin{aligned}
        s^3\|e^{s\ffi}\chi_k  &\widetilde{v}\|^2_{L^2(\qc_k)}
        - 2s\|e^{s\ffi}\widetilde{v}D_h^+\chi_k \|^2_{L^2(\qc_k)}
        - 2s\|e^{s\ffi}\widetilde{v}D_h^-\chi_k \|^2_{L^2(\qc_k)}
        -\|e^{s\ffi}\widetilde{v}\Delta_h\chi_k \|^2_{L^2(\qc_k)}\\
        &+s\|e^{s\ffi}\chi_k D_h^+\widetilde{v}\|^2_{L^2(\qc_k)}
        -3\|e^{s\ffi}D_h^+\chi_k D_h^+\widetilde{v}\|^2_{L^2(\qc_k)} \\
        &+s\|e^{s\ffi}\chi_k D_h^-\widetilde{v}\|^2_{L^2(\qc_k)}
        -3\|e^{s\ffi}D_h^-\chi_k D_h^-\widetilde{v}\|^2_{L^2(\qc_k)} \\
        &+s^3e^{2s\ffi(3)}\|\chi_k \widetilde{v}(3,\cdot)\|^2_{\ell^2(Q_k)}
        -se^{2s\ffi(3)}\|\widetilde{v}(3,\cdot)D_h^+\chi_k \|^2_{\ell^2(Q_k)}
        -se^{2s\ffi(3)}\|\widetilde{v}(3,\cdot)D_h^-\chi_k \|^2_{\ell^2(Q_k)}\\       
        \lesssim& +se^{2s\ffi(3)}\|(\chi_k+hD_h^+\chi_k)D_h^+\widetilde{v}(3,\cdot)\|^2_{\ell^2(Q_k)} 
        +se^{2s\ffi(3)}\|(\chi_k-hD_h^-\chi_k)D_h^-\widetilde{v}(3,\cdot)\|^2_{\ell^2(Q_k)} \\
        &+s\|\chi_ke^{s\ffi(0,\cdot)}\partial_t \widetilde{v}(0,\cdot)\|^2_{\ell^2(\omega_k)}\\
        \lesssim& se^{2s\ffi(3)}\|D_h^+\widetilde{v}(3,\cdot)\|^2_{\ell^2(Q_k)}
        + se^{2s\ffi(3)}\|D_h^-\widetilde{v}(3,\cdot)\|^2_{\ell^2(Q_k)}
        +s\|e^{s\ffi(0,\cdot)}\partial_t \widetilde{v}(0,\cdot)\|^2_{\ell^2(\omega_k)}
    \end{aligned}
    \end{equation}
    since $\chi_k\leq1$.
Replacing $s_0$ by a larger constant in \eqref{splusgrand:eq1} allows us to absorb all the negative terms into the positive terms in the same line, leading us to
\begin{equation*}
      \begin{aligned}
        s^3\|e^{s\ffi}\chi_k  \widetilde{v}\|^2_{L^2(\qc_k)}&+s\|e^{s\ffi}\chi_k D_h^+\widetilde{v}\|^2_{L^2(\qc_k)}
        +s\|e^{s\ffi}\chi_k D_h^-\widetilde{v}\|^2_{L^2(\qc_k)}+s^3e^{2s\ffi(3)}\|\chi_k \widetilde{v}(3,\cdot)\|^2_{\ell^2(Q_k)}\\
        \lesssim&
        se^{2s\ffi(3)}\|D_h^+\widetilde{v}(3,\cdot)\|^2_{\ell^2(Q_k)}+se^{2s\ffi(3)}\|D_h^-\widetilde{v}(3,\cdot)\|^2_{\ell^2(Q_k)}
        +s\|e^{s\ffi(0,\cdot)}\partial_t \widetilde{v}(0,\cdot)\|^2_{\ell^2(\omega_k)}.
    \end{aligned}
    \end{equation*}
Finally, we keep only the last term on the left 
and shrink the norm from $\ell^2(Q_k)$ to $\ell^2(kL+Q_{L})$ and get 
 \begin{multline*}
s^3e^{2s\ffi(3)}\|\widetilde{v}(3,\cdot)\|^2_{\ell^2(kL+Q_{L})} \\ \lesssim 
se^{2s\ffi(3)}\|D_h^+\widetilde{v}(3,\cdot)\|^2_{\ell^2(Q_k)}+se^{2s\ffi(3)}\|D_h^-\widetilde{v}(3,\cdot)\|^2_{\ell^2(Q_k)}+s\|e^{s\ffi(0,\cdot)}\partial_t \widetilde{v}(0,\cdot)\|^2_{\ell^2(\omega_k)}.
\end{multline*}
We now sum over $k$ and use that the cubes $Q_k$ have only finite overlap to obtain
\begin{multline*}
s^3e^{2s\ffi(3)}\|\widetilde{v}(3,\cdot)\|^2_{\ell^2(h\Z^d)}  \\
    \lesssim se^{2s\ffi(3)}\|D_h^+\widetilde{v}(3,\cdot)\|^2_{\ell^2(h\Z^d)}+se^{2s\ffi(3)}\|D_h^-\widetilde{v}(3,\cdot)\|^2_{\ell^2(h\Z^d)}+s\|e^{s\ffi(0,\cdot)}\partial_t \widetilde{v}(0,\cdot)\|^2_{\ell^2(\omega)}.
\end{multline*}
Note that on $h\Z^d$ we have that $\|D_h^+\widetilde{v}(3,\cdot)\|^2_{\ell^2(h\Z^d)}=\|D_h^-\widetilde{v}(3,\cdot)\|^2_{\ell^2(h\Z^d)}$, then
\begin{equation}
s^3e^{2s\ffi(3)}\|\widetilde{v}(3,\cdot)\|^2_{\ell^2(h\Z^d)} 
    \lesssim se^{2s\ffi(3)}\|D_h^+\widetilde{v}(3,\cdot)\|^2_{\ell^2(h\Z^d)}+s\|e^{s\ffi(0,\cdot)}\partial_t \widetilde{v}(0,\cdot)\|^2_{\ell^2(\omega)}.
\label{a-9}
\end{equation}
The first term in the right-hand side of~\eqref{a-9} is bounded from above:
\begin{equation*}
    \begin{aligned}
       se^{2s\ffi(3)} \|D_h^+\widetilde{v}(3,\cdot)\|^2_{\ell^2(h\Z^d)}&= - se^{2s\ffi(3)}\langle \Delta_h \widetilde{v}(3,\cdot), \widetilde{v}(3,\cdot)  \rangle_{\ell^2(h\Z^d)}\\
        &=se^{2s\ffi(3)} \langle \partial_t^2 \widetilde{v}(3,\cdot),\widetilde{v}(3,\cdot) \rangle_{\ell^2(h\Z^d)}-se^{2s\ffi(3)}\langle V \widetilde{v}(3,\cdot),\widetilde{v}(3,\cdot) \rangle_{\ell^2(h\Z^d)} \\
        &\lesssim se^{2s\ffi(3)}(\mu+\|V\|_\infty) \|\widetilde{v}(3,\cdot)\|^2_{\ell^2(h\Z^d)}
    \end{aligned}
\end{equation*}
from the definition of $\widetilde{v}(3,\cdot)$.
Taking this estimate into \eqref{a-9}, and observing that $\partial_t \widetilde{v}(0,\cdot)=v$, we obtain
\begin{equation}
s^3e^{2s\ffi(3)}\|\widetilde{v}(3,\cdot)\|^2_{\ell^2(h\Z^d)}\lesssim s (\mu+\|V\|_\infty) e^{2s\ffi(3)}\|\widetilde{v}(3,\cdot)\|^2_{\ell^2(h\Z^d)}+ s  \| e^{s\ffi(0,\cdot)} v\|^2_{\ell^2(\omega)}.\label{a-10}
\end{equation}
Taking 
\begin{equation}\label{s2plusgrand:eq1}
s\geq s_0(\mu+1+\|V\|_\infty)^{\frac{1}{2}}
\end{equation}
after eventually increasing $s_0$,
we absorb the first term on the right into the left.
Finally as 
$$
\|\widetilde{v}(3,\cdot)\|^2_{\ell^2(h\Z^d)}\gtrsim\|v\|^2_{\ell^2(h\Z^d)},
$$
we obtain
\begin{equation*}
s^3e^{2s\ffi(3)}\|v\|^2_{\ell^2(h\Z^d)}\lesssim s  \| e^{s\ffi(0,\cdot)} v\|^2_{\ell^2(\omega)}.
\end{equation*}
As $s^3e^{2s\ffi(3)}\geq s$ and $e^{s\ffi(0,\cdot)}\lesssim e^s$, we conclude that
$\|v\|^2_{\ell^2(h\Z^d)}\lesssim e^s\|v\|^2_{\ell^2(\omega)}$. One can choose 
\begin{equation*}
    s=2s_0\bigl(\mu+1+\|V\|_\infty^{\frac{4}{3}}\bigr)^{\frac{1}{2}},
\end{equation*}
then conditions \eqref{splusgrand:eq1} and \eqref{s2plusgrand:eq1} on $s$ are satisfied.

If $V\in C^1_b$, condition \eqref{splusgrand:eq1} is replaced by $s\gtrsim s_0 (1+\|V\|_{W^{1,\infty}})^{\frac{1}{2}}$ and we can simply choose $s= s_0(\mu+1+\|V\|_{W^{1,\infty}})^{\frac{1}{2}}$, giving the second spectral inequality.
\end{proof}

\subsection{Proof of spectral inequalities for power growth potentials}
In this section, we give the proof of spectral inequalities for potentials satisfying Assumption~\ref{assump}
that in particular contains the case $V(x)=|x|^\beta$ stated in Theorem \ref{th:main:1} in the introduction,
which corresponds to $\beta_1=\beta_2=\beta$.

\begin{theorem}
Let $\omega\subset\R^d$ be equidistributed. Let $V$ be a potential satisfying Assumption~\ref{assump}
and $P_h=-\Delta_h+V$ with associated spectral projector $\Pi_{\mu,h}$.
Then there exists a constant $C>0$, $\varepsilon_0>0$ and $h_0>0$ depending only on $d,\omega,c_1,c_2,\beta_1,\beta_2$ 
in Assumption~\ref{assump} such that, for any $h<h_0$ and $0<\mu<\varepsilon_0 / h^{\frac{2\beta_1}{\beta_2}}$, we have
    \begin{equation}
    \label{eq:specwithpotA}
        \|\Pi_{\mu,h} u\|_{\ell^2(h\Z^d)}^2 \le Ce^{C\mu^{\frac{\beta_2}{2\beta_1}}}\|\Pi_{\mu,h} u\|^2_{\ell^2(\omega)},\quad \forall u\in \ell^2(h\Z^d).
    \end{equation}
\end{theorem}

\begin{proof}
    Recall that there exists an orthonormal sequence in $\ell^2(h\Z^d)$ of eigenfunctions $\lbrace\phi_{k}\rbrace_{k\in\N}$ and the corresponding sequence of eigenvalues $\lbrace \lambda_{k}\rbrace_{k\in\N}$ such that 
\begin{equation*}
    (-\Delta_h+V) \phi_k=\lambda_k\phi_k,\quad \forall k\in\N.
\end{equation*}
Note also that $\lambda_k\geq 0$ since $V$ is non-negative.

    The first step is still to lift the dimension, while the representation is a little bit different. Let $v\in \mathcal{E}_\mu (P_h)$, written in the form 
    \begin{equation*}
        v=\sum_{\lambda_k\le \mu}c_k \phi_k, \quad c_k\in\C.
    \end{equation*}
    Define
    \begin{equation}
        \widetilde{v}(t,x):=\sum_{\lambda_k\le \mu} c_k\frac{\sinh (\sqrt{\lambda_k} t)}{\sqrt{\lambda_k}} \phi_k.\label{a-16}
    \end{equation}
    Then one can verify directly that
    \begin{equation*}
        (\partial_t^2-P_h)\widetilde{v}=0
    \end{equation*}
    and 
    \begin{equation}
        \partial_t \widetilde{v}|_{t=0} =v.\label{a-15}
    \end{equation}

As $\omega$ is $(L,\gamma)$-equidistributed, for every $k\in \Z^d$, $\omega\cap(Lk+[-L/2,L/2]^d)\supset Q_{\gamma L}(x_k)$.
Without loss of generality, we may shrink $\omega$ so that $\omega=\displaystyle\bigcup_{k\in\Z^d}\omega_k$, with 
$\omega_k= Q_{\gamma L}(x_k)$. We then cover $\R^d=\bigcup_{k\in\Z^d} Q_k$ with $Q_k:=Q_{6L}(x_k)$.

Thanks to Proposition~\ref{prp-localization}, we have
$$
\|\Pi_{\mu,h} u\|_{\ell^2(h\Z^d)}^2\lesssim \|\Pi_{\mu,h} u\|_{\ell^2(Q_{L_\mu})}^2\le N_d\sum_{k\in \mathcal{J}_\mu}\|\Pi_{\mu,h} u\|_{\ell^2(Q_k)}^2
$$
where $N_d$ is the finite overlap constant and
    \begin{equation*}
    L_\mu \simeq \mu^{\frac{1}{\beta_1}} \quad\mbox{and}\quad
        \mathcal{J}_\mu:=\lbrace k\in\Z^d: Q_{3L}(x_k)\cap Q_{L_\mu}\neq \varnothing \rbrace.
    \end{equation*}

We now fix  $k\in\mathcal{J}_\mu$ and define $u(t,x)=\chi(x-x_k)\widetilde{v}(t,x)$
    where $\chi$ is the smooth nonnegative function given in \eqref{chi:q} and apply the Carleman estimate in Theorem~\ref{crc:Carleman:1}. This requires
    \begin{equation*}
        s\ge s_0(\|V_1\|_{W^{1,\infty}(Q_{2L_\mu})}^{\frac{1}{2}}+\|V_2\|_{L^\infty(Q_{2L_\mu})}^{\frac{2}{3}}+1)
    \end{equation*}
    for some positive $s_0$. From Assumption~\ref{assump}, we get 
    $$
    \|V_1\|_{W^{1,\infty}(Q_{2L_\mu})}^{\frac{1}{2}}+\|V_2\|_{L^\infty(Q_{2L_\mu})}^{\frac{2}{3}}\lesssim L_\mu^{\frac{\beta_2}{2}}\lesssim \mu^{\frac{\beta_2}{2\beta_1}}.
    $$
We will thus require that
\begin{equation}
        s\ge s_0 (\mu+1)^{\frac{\beta_2}{2\beta_1}}\label{a-12}
    \end{equation}
    for some large enough $s_0:=s_0(c_1,c_2)$. Note that we can always enlarge $s_0$ to be able to apply the Carleman estimate. As in the previous section, after absorbing all the negative terms by choosing $s_0$ large enough, this gives 
\begin{align*}
s^3e^{2s\ffi(3)}\|\chi(\cdot-x_k) \widetilde{v}(3,\cdot)\|^2_{\ell^2(Q_{k})} 
\lesssim& se^{2s\ffi(3)}\|\chi(\cdot-x_k) D_h^+\widetilde{v}(3,\cdot)\|^2_{\ell^2(Q_{k})}
 +s\|e^{s\ffi(0,\cdot)}\partial_t \widetilde{v}(0,\cdot)\|^2_{\ell^2(\omega_k)} \\
&+se^{2s\ffi(3)}\|\chi(\cdot-x_k) D_h^-\widetilde{v}(3,\cdot)\|^2_{\ell^2(Q_{k})}
\end{align*}
with the condition \eqref{a-12} for $s_0$ sufficiently large. Summing them over $\mathcal{J}_\mu$, we obtain
\begin{equation}
\begin{aligned}
    &\underbrace{\sum_{k\in\mathcal{J}_\mu} s^3e^{2s\ffi(3)}\|\chi(\cdot-x_k) \widetilde{v}(3,\cdot)\|^2_{\ell^2\left(Q_{k}\right)}}_{\mathrm{LHS}} \\ 
    \lesssim &\underbrace{\sum_{k\in\mathcal{J}_\mu}\!se^{2s\ffi(3)}\|\chi(\cdot-x_k) D_h^+\widetilde{v}(3,\cdot)\|^2_{\ell^2\left(Q_{k}\right)}+\sum_{k\in\mathcal{J}_\mu}\!se^{2s\ffi(3)}\|\chi(\cdot-x_k) D_h^-\widetilde{v}(3,\cdot)\|^2_{\ell^2\left(Q_{k}\right)}}_{\mathrm{RHS}_1}  \\
    &+\sum_{k\in\mathcal{J}_\mu}\!s\|e^{s\ffi(0,\cdot)}\partial_t \widetilde{v}(0,\cdot)\|^2_{\ell^2(\omega_k)}.\label{a-13}
\end{aligned}
\end{equation}

One can bound the LHS from below
\begin{equation*}
    \mathrm{LHS}\gtrsim s^3 e^{2s\ffi(3)}\| \widetilde{v}(3,\cdot) \|^2_{\ell^2(Q_{L_\mu})}
    \gtrsim s^3 e^{2s\ffi(3)}\| \widetilde{v}(3,\cdot) \|^2_{\ell^2(h\Z^d)}
\end{equation*}
where we have used the localization property, i.e., Proposition \ref{prp-localization} by viewing $\widetilde{v}(3,\cdot)$ as a linear combination of eigenfunctions in the lattice. Then by the orthogonality of the normalized eigenfunctions, we obtain
\begin{equation}\label{lhs-unbounded}
    \mathrm{LHS}\gtrsim s^3 e^{2s\ffi(3)}\sum_{\lambda_k\le \mu} \frac{\sinh^2(3\sqrt{\lambda_k})}{\lambda_k} |c_k|^2.
\end{equation}
For the terms in the right-hand side of \eqref{a-13} involving discrete derivatives $\mathrm{RHS}_1$, since by \eqref{Summation:eq2} we have
\begin{align*}
    \| D_h^\pm \widetilde{v}(3,\cdot) \|^2_{\ell^2(h\Z^d)}
    &=\biggl\langle \sum_{\lambda_k\le \mu} c_k\frac{\sinh(3\sqrt{\lambda_k})}{\sqrt{\lambda_k}} D_h^\pm\phi_k,\sum_{\lambda_k\le \mu} c_k\frac{\sinh(3\sqrt{\lambda_k})}{\sqrt{\lambda_k}}  D_h^\pm\phi_k\biggr\rangle_{\ell^2(h\Z^d)}\\
    &=-2\biggl\langle \sum_{\lambda_k\le \mu} c_k\frac{\sinh(3\sqrt{\lambda_k})}{\sqrt{\lambda_k}} \Delta_h\phi_k,\sum_{\lambda_k\le \mu} c_k\frac{\sinh(3\sqrt{\lambda_k})}{\sqrt{\lambda_k}}  \phi_k\biggr\rangle_{\ell^2(h\Z^d)}\\
    &\le   2\biggl\langle \sum_{\lambda_k\le \mu} c_k\frac{\sinh(3\sqrt{\lambda_k})}{\sqrt{\lambda_k}} (-\Delta_h+V)\phi_k,\sum_{\lambda_k\le \mu} c_k\frac{\sinh(3\sqrt{\lambda_k})}{\sqrt{\lambda_k}}  \phi_k\biggr\rangle_{\ell^2(h\Z^d)}\\
    & = 2\sum_{\lambda_k\le \mu}\sinh^2(3\sqrt{\lambda_k}) |c_k|^2
    \leq2\mu \sum_{\lambda_k\le \mu}\frac{\sinh^2(3\sqrt{\lambda_k})}{{\lambda_k}} |c_k|^2
\end{align*}
where we used that $V\geq0$ and orthonormality of the $\phi_k$'s. Then
\begin{equation*}
    s^{-1}e^{-2s\ffi(3)}\,\mathrm{RHS}_1\lesssim \| D_h^+ \widetilde{v}(3,\cdot) \|^2_{\ell^2(h\Z^d)} + \| D_h^- \widetilde{v}(3,\cdot) \|^2_{\ell^2(h\Z^d)}
    \lesssim \mu \sum_{\lambda_k\le \mu}\frac{\sinh^2(3\sqrt{\lambda_k})}{{\lambda_k}} |c_k|^2.
\end{equation*}
Finally we obtain
\begin{equation*}
    s^3e^{2s\ffi(3)}\sum_{\lambda_k\le \mu}\frac{\sinh^2(3\sqrt{\lambda_k})}{\lambda_k}|c_k|^2
    \lesssim s e^{2s\ffi(3)}\mu\sum_{\lambda_k\le \mu}\frac{\sinh^2(3\sqrt{\lambda_k})}{\lambda_k}|c_k|^2
    + s  \| e^{s\ffi(0,\cdot)} v\|^2_{\ell^2(\omega)}.
\end{equation*}
Since $s$ satisfies \eqref{a-12} and $\dfrac{\beta_2}{2\beta_1}\ge \dfrac{1}{2}$
we have $s^2\ge s_0^2 \mu$
thus, taking $s_0$ large enough, the first term in the right-hand side can be absorbed in the left hand side.  This gives
$$
s^2e^{2s\ffi(3)}\sum_{\lambda_k\le \mu}\frac{\sinh^2(3\sqrt{\lambda_k})}{\lambda_k}|c_k|^2
\lesssim e^{\kappa s}\|v\|_{\ell^2(\omega)}^2
$$
with $\kappa=\sup \phi(0,\cdot)<+\infty$.
As $s^2e^{2s\ffi(3)}\gtrsim 1$ and $\dfrac{\sinh^2(3\sqrt{t})}{t}\gtrsim 1$, we get
\begin{equation*}
    \sum_{\lambda_k\le \mu}|c_k|^2
    \lesssim e^{\kappa s} \|v\|_{\ell^2(\omega)}^2
\end{equation*}
which is the claimed inequality. 
\end{proof}

\section{Controllability of the lower part of the spectrum}\label{sec:control}
In this section, we give the proofs of the controllability results that we deduce from spectral inequalities established in 
the previous section. This follows directly from \cite[\S 7]{boye2010disc} so that we will not reproduce the argument
when $V$ satisfies Assumption~\ref{assump}. However, when $V\in C_b(\R^d)$, it is possible to slightly improve the 
reasoning in order to keep track of the explicit dependence of the constants on the potential $V$. 
Further, the observability inequalities directly follow from the arguments given for instance in \cite{GL94,LT06}
so that we do not reproduce the argument here.

\subsection{Bounded non-negative potentials}

Define 
\begin{equation*}
    E_j:=\mathcal{E}_{2^{2j}}(P_h) \quad \text{and}\quad \Pi_{E_j}:= \Pi_{2^{2j},h}
\end{equation*}
where we used the notations given in \eqref{def-project} and \eqref{def-subspace}. We also 
define $J_h=\left\lfloor\dfrac{\log_2 (\eps_0 h^{-2})}{2}\right\rfloor$ where $\eps_0$ is given in Theorem \ref{th:main:1}.
In particular $2^{2J_h}\simeq h^{-2}$.
The first step consists in proving a final time observability inequality for low-frequency functions:

\begin{lemma}\label{adjoint-obs-system-lma}
    Let $V\in C_b(\R^d)$ or $V\in C_b^1(\R^d)$, $V\geq 0$, and $\omega$ be equidistributed. There exists $C,\kappa>0$ such that, for $j\le J_h$ and $T>0$, the semi-discrete solution $v\in C^\infty ([0,T], E_j)$ to the adjoint parabolic system
    \begin{equation}\label{adjoint-obs-system}
        \begin{cases}
            \partial_t v+\Delta_h v -Vv=0,& \text{in } (0,T)\times h\Z^d,\\
            v(T)=v_F\in  E_j,
        \end{cases}
    \end{equation}
    satisfies the observability estimate
    \begin{equation*}
        \|v(0)\|^2_{\ell^2(h\Z^d)}\le C_{\mathrm{obs}}(T)\int_0^T\|v(t)\|^2_{\ell^2(\omega)}\d t
    \end{equation*}
    where the constant $C_{\mathrm{obs}}(T)$ is given by
    \begin{equation}\label{form-obs-const}
        C_{\mathrm{obs}}(T)=\begin{cases}
            CT^{-1}e^{\kappa(1+\|V\|^{2/3}_{L^\infty})2^j},& \text{for } V\in C_b(\R^d),\\
            CT^{-1}e^{\kappa(1+\|V\|_{W^{1,\infty}}^{1/2})2^j},& \text{for } V\in C_b^1(\R^d).
        \end{cases}
    \end{equation}
\end{lemma}

\begin{proof}
    Since $v_F\in E_j$, we have $v(t)\in E_j$ for any $t\in [0,T]$. Then we write
    \begin{equation*}
        v(t)=e^{-(T-t)(-\Delta_h+V)} v_F.
    \end{equation*}
    For the case $V\in C_b(\R^d)$, by parabolic dissipation and the partial spectral inequality \eqref{ineq-bounded-only} in Theorem \ref{th:main:1}, we have
    \begin{equation*}
        \begin{aligned}
            T\|v(0)\|^2_{\ell^2(h\Z^d)}&\le \int_0^T \|v(t)\|^2_{\ell^2(h\Z^d)} \d t\\
            &\le C e^{\kappa\sqrt{1+\|V\|^{4/3}_{L^\infty}+2^{2j}}}\int_0^T \|v(t)\|^2_{\ell^2(\omega)}\d t.
        \end{aligned}
    \end{equation*}
    Finally, we rewrite the result using that $\sqrt{1+\|V\|^{4/3}_{L^\infty}+2^{2j}}\leq(1+\|V\|^{2/3}_{L^\infty})2^j$.
    
    The case of $V\in C_b^1(\R^d)$ can be obtained in the same way by using \eqref{inequ-c1-bounded} instead.
\end{proof}

\begin{remark}
    All estimates given below will depend on the estimate of $C_{\mathrm{obs}}(T)$ given here.
    In particular, when $V\in C^1_b(\R^d)$, one can replace $\|V\|^{2/3}_{L^\infty}$ with 
    $\|V\|_{W^{1,\infty}}^{1/2}$. We will therefore focus on the case $V\in C_b(\R^d)$ for the remaining of the proof.
\end{remark}

Next we consider the following partial control problem
\begin{equation}\label{adjoint-partial-system}
    \begin{cases}
        \partial_t u-\Delta_h u+Vu=\Pi_{E_j}(\un_\omega f),& \text{in } (0,T)\times h\Z^d,\\
        u(0)=u_0\in E_j.
    \end{cases}
\end{equation}
From the previous partial observability result we obtain the following equivalent lemma (see, e.g., \cite[Proposition~7.7]{lerousseau2022elliptic} for their equivalence or \cite[Theorem 11.2.1]{tucsnak2009observation} for a more abstract version).

\begin{lemma}\label{control-lma}
Let $V\in C_b(\R^d)$ or $V\in C_b^1(\R^d)$, $V\geq0$, and $\omega$ be equidistributed. 
There exists $C,\kappa>0$ such that for $j\le J_h$ and $T>0$, there exists a control function $f\in L^2((0,T), \ell^2(h\Z^d))$ that brings the solution of \eqref{adjoint-partial-system} to zero at time $T$, and which satisfies
    \begin{equation*}
        \int_0^T\| f(t)\|^2_{\ell^2(h\Z^d)}\d t \le C_{\mathrm{obs}}(T) \|u_0\|^2_{\ell^2(h\Z^d)},
    \end{equation*}
    where the form of $C_{\mathrm{obs}}(T)$ is the one given in \eqref{form-obs-const}.
\end{lemma}

Now we are in the position to prove Theorem \ref{thm-uniform-control}.

\begin{proof}[Proof of Theorem \ref{thm-uniform-control}]
    Assume that $V\in C_b(\R^d)$. We present the iterative construction of the control function $f$ with the desired constants. We denote by $W_j(u_0,a,\tau)$ the control function given by Lemma \ref{control-lma} on the time interval $(a,a+\tau)$.

   Take $\rho\in(0,1)$ a parameter.  We split the interval $[0,T/2]=\bigcup_{j\in \N} [a_j,a_{j+1}]$, with $a_0=0$, $a_{j+1}=a_j+2T_j$, for any $j\in \N$ and $T_j=L2^{-j\rho}$ and the constant $L$ chosen so that 
    $$
    2\sum_{j=0}^\infty T_j=\frac{T}{2}.
    $$
    This gives the relation 
    \begin{equation}
        L=\frac{1-2^{-\rho}}{4}T\label{t-relation-rho}
    \end{equation}
    for a fixed $\rho \in (0,1)$.
    Further let $S(t)$ denote the semigroup $S(t):= e^{-t(-\Delta_h+V)}$. 
    In particular we have 
    \begin{equation}\label{semigroup:eq1}
    \|S(t)\|_{L^2(h\Z^d)\to L^2(h\Z^d)}\le 1.
    \end{equation}
    We define the control function $f$ for \eqref{eq:heat} as follows: for $0\le j\le J_h$,
    \begin{itemize}
        \item if $t\in (a_j,a_j+T_j]$, then $f(t,x)=W_j\bigl(\Pi_{E_j} u(a_j),a_j,T_j\bigr)$ and 
        \begin{equation*}
            u(t)=S(t-a_j)u(a_j)+\int_{a_j}^t S(t-s)\un_{\omega}f(s)\d s;
        \end{equation*}
    \item if $t \in (a_j+T_j, a_{j+1}]$, then $f(t,x)=0$ and $u(t)=S(t-a_j-T_j) u(a_j+T_j)$,
    \item if $t\in [a_{J_h+1},T]$, then $f(t)=0$.
    \end{itemize} 

First, by construction, $u(T)\in E_{J_h}^{\perp}$. As $2^{2J_h}\simeq h^{-2}$, 
$\Pi_{C_3/h^2,h}u(T)=0$ for some $C_3>0$.

Note that, from Lemma \ref{control-lma},
\begin{equation}
    \label{eq:estcontcostsi}
\|f\|^2_{L^2((a_j,a_j+T_j); \ell^2(h\Z^d))}
\leq \frac{C}{T_j}e^{\kappa (1+\|V\|_{\infty}^{2/3})2^j}\|u(a_j)\|_{\ell^2(h\Z^d)}^2
\end{equation}
In particular, for $j=0$, we have
\begin{equation}
    \|f\|^2_{L^2((0,T_0); \ell^2(h\Z^d))}\le \frac{4C}{(1-2^{-\rho})T} \|u_0\|_{\ell^2(h\Z^d)}^2.\label{eq:estcontcostsi2}
\end{equation}

Now, during the period $(a_j,a_j+T_j]$, this choice of the control function $f=W_j\bigl(\Pi_{E_j} u(a_j),a_j,T_j\bigr)$ and \eqref{semigroup:eq1} implies
    \begin{equation}
        \begin{aligned}
            \|u(a_{j}+T_j)\|_{\ell^2(h\Z^d)}
            &\le \|u(a_j)\|_{\ell^2(h\Z^d)}
            + \int_{a_j}^{a_j+T_j}\|f(s,\cdot)\|_{\ell^2(\omega)}\d s\\
            &\le \|u(a_j)\|_{\ell^2(h\Z^d)}
            +T_j^{\frac{1}{2}}\left(\int_{a_j}^{a_j+T_j}\|f(s,\cdot)\|^2_{\ell^2(h\Z^d)}\d s\right)^{\frac{1}{2}}\\
            &\le \biggl(1+Ce^{\kappa(1+\|V\|_{\infty}^{2/3})2^{j-1/2}}\biggr)\|u(a_j)\|_{\ell^2(h\Z^d)} 
            \label{eq:aaaaa1}
        \end{aligned}
    \end{equation}
    where we used the Cauchy--Schwarz inequality in the second identity and \eqref{eq:estcontcostsi} in the last one. Furthermore, we have $\Pi_{E_j}u(a_{j}+T_j)=0$ by the construction of our control function.
    During the period $(a_j+T_j,a_{j+1}]$, as $f=0$ on this interval, by parabolic dissipation, there is an exponential decrease of the $\ell^2$-norm, that is,
    \begin{equation}
    \label{eq:aaaaa2}
        \|u(a_{j+1})\|_{\ell^2(h\Z^d)} 
        \le e^{-2^{2j}T_j}\| u(a_j+T_j) \|_{\ell^2(h\Z^d)}
        \le e^{-L2^{j(2-\rho)}}\| u(a_j+T_j) \|_{\ell^2(h\Z^d)}
    \end{equation}
    where we used the fact that only the spectrum larger than $2^{2j}$ is left during the period $(a_j+T_j,a_{j+1}]$.

    For the same reason, as $2^{2J_h}\simeq h^{-2}$, there is a constant $C_2$ such that
\begin{equation}
    \label{eq:estuT}
\|u(T)\|_{\ell^2(h\Z^d)} \leq e^{-2C_2(T-a_{J_h+1})/h^2}\| u(a_{J_h+1}) \|_{\ell^2(h\Z^d)}
\leq e^{-C_2T/h^2}\| u(a_{J_h+1}) \|_{\ell^2(h\Z^d)}.
\end{equation}

Now, combining \eqref{eq:aaaaa1} and \eqref{eq:aaaaa2}, we obtain
\begin{equation*}
    \|u(a_{j+1})\|_{\ell^2(h\Z^d)}\leq
    Ce^{\kappa (1+\|V\|^{\frac{2}{3}}_{L^\infty})2^{j-1/2}-L2^{j(2-\rho)}}\|u(a_j)\|_{\ell^2(h\Z^d)}
\end{equation*}

Iterating this gives
\begin{align}
        \|u(a_{j+1})\|_{\ell^2(h\Z^d)}
        &\le C \prod_{k=0}^j \exp{\Bigl(\kappa(1+\|V\|^{\frac{2}{3}}_{L^\infty})2^{k-1/2}-L2^{k(2-\rho)}\Bigr)}
        \|u_0\|_{\ell^2(h\Z^d)} \notag\\
        &\le C \exp{\left(\kappa \left(1+\|V\|_{L^\infty}^{\frac{2}{3}}\right) 2^{j+1/2}
        -(2^{2-\rho}-1)^{-1}L2^{(j+1)(2-\rho)}\right)} \|u_0\|_{\ell^2(h\Z^d)}\notag\\
        &=C \exp{\left(\kappa \left(1+\|V\|_{L^\infty}^{\frac{2}{3}}\right) 2^{j+1/2}
        -c_\rho T2^{(j+1)(2-\rho)}\right)} \|u_0\|_{\ell^2(h\Z^d)}. \label{estimate-j-step} 
\end{align}
with
\begin{equation}
    c_\rho=\dfrac{1-2^{-\rho}}{4(2^{2-\rho}-1)}\label{eq:def-of-crho}
\end{equation}
in view of \eqref{t-relation-rho}.
Together with \eqref{eq:estcontcostsi}, this implies that
\begin{align}
\|f\|^2_{L^2((a_j,a_j+T_j); \ell^2(h\Z^d))}
&\leq \frac{C2^{j\rho}}{(1-2^{-\rho})T}\exp{\left(3\kappa \left(1+\|V\|_{L^\infty}^{\frac{2}{3}}\right) 2^{j}
        -c_\rho T2^{j(2-\rho)}\right)} \|u_0\|_{\ell^2(h\Z^d)}^2\notag\\
&\leq \frac{C}{(1-2^{-\rho})T}\exp{\left(6\kappa \left(1+\|V\|_{L^\infty}^{\frac{2}{3}}\right) 2^{j}
        -c_\rho T2^{j(2-\rho)}\right)} \|u_0\|_{\ell^2(h\Z^d)}^2
        \label{eq:estcontcostsifull}
\end{align}
where the factor $2^{j\rho}\leq 2^{2j}$ is absorbed in the exponential by replacing $\kappa$ with $2\kappa$.
Note also that $C,\kappa$ are independent of $\rho$.

We now split the analysis into two cases.

\smallskip

\noindent{\bf Case 1: large time, $T\ge K_\rho\bigl(1+\|V\|^{\frac{2}{3}}_{L^\infty}\bigr)$ with $K_\rho:=\dfrac{7\kappa}{c_\rho}$.}

\smallskip

In this case, the positive terms
in the exponentials in \eqref{estimate-j-step}-\eqref{eq:estcontcostsifull} can be absorbed in the negative ones.
In other words, \eqref{estimate-j-step} now reads 
\begin{equation*}
    \|u(a_j)\|_{\ell^2(h\Z^d)}\le 
    C e^{-6\kappa (1+\|V\|^{\frac{2}{3}}_{L^\infty})2^{2j(1-\rho/2)}} \|u_0\|_{\ell^2(h\Z^d)}
    \leq \|u_0\|_{\ell^2(h\Z^d)}\quad,\ 1\leq j\leq J_h.
\end{equation*}
In particular, \eqref{eq:estuT} implies that 
\begin{equation*}
    \|u(T)\|\leq  e^{-C T/h^{2}} \|u(a_{J_h+1})\|_{\ell^2(h\Z^d)}\le 
    C e^{-C T/h^{2}}\|u_0\|_{\ell^2(h\Z^d)}.
\end{equation*}
Further, from the condition on $T$ and $0<\rho<1$, \eqref{eq:estcontcostsifull} reads
\begin{equation}
\|f\|^2_{L^2((a_j,a_j+T_j); \ell^2(h\Z^d))}
\le \frac{2C}{(1-2^{-\rho})T}e^{-c_\rho T 2^{j}/7} \|u_0\|_{\ell^2(h\Z^d)}^2
\label{eq:estcontcostsifulllargeT}
\end{equation}
This together with \eqref{eq:estcontcostsi2} leads to the estimate of the norm of $f$:
\begin{equation*}
\begin{aligned}
    \|f\|^2_{L^2((0,T); \ell^2(h\Z^d))}
&=\sum_{j=0}^{J_h}\|f\|^2_{L^2((a_j,a_j+T_j); \ell^2(h\Z^d))}\\
    &\le \left(\frac{4C}{(1-2^{-\rho})T}+ \frac{C}{(1-2^{-\rho})T} \sum_{j=1}^{J_h}  e^{-c_\rho T2^j/7}\right)
    \|u_0\|^2_{\ell^2(h\Z^d)}\\
    &\le\left(\frac{4C}{(1-2^{-\rho})T}+  \frac{C}{1-2^{-\rho}} \frac{e^{-c_\rho T/14}}{T}\right)\|u_0\|^2_{\ell^2(h\Z^d)}\\
    &\le \frac{5C}{(1-2^{-\rho})T}
    \end{aligned}
\end{equation*}
where in the third line we used
$$
\sum_{j=0}^{+\infty} e^{-\alpha 2^j}\leq Ce^{-\alpha/2}.
$$
Then the theorem is established in this case.

\smallskip
\noindent{\bf Case 2: small time, $T\le \dfrac{7\kappa}{c_\rho}\bigl(1+\|V\|^{\frac{2}{3}}_{L^\infty}\bigr)$.}

\smallskip

In this case we cannot simply absorb the positive terms in the exponential in the negative ones as before.
To overcome this, we introduce $g(x)=6\kappa \left(1+\|V\|_{L^\infty}^{\frac{2}{3}}\right) x-c_\rho Tx^{2-\rho}$
on $(0,+\infty)$
and notice that $g$ vanishes only at
\begin{equation}
    x_1=\left(\dfrac{6\kappa \left(1+\|V\|_{L^\infty}^{\frac{2}{3}}\right)}{c_\rho T}\right)^{\frac{1}{1-\rho}}.
    \label{x1-rho}
\end{equation}
Let $j_*$ be the minimum value such that $2^{j_*}\ge x_1$, then
\begin{equation}
    j_*\simeq \log_2 \left(\dfrac{6\kappa \left(1+\|V\|_{L^\infty}^{\frac{2}{3}}\right)}{c_\rho T}\right)^{\frac{1}{1-\rho}}.\label{bound-j-starnew}
\end{equation}
As $J_h\simeq \frac{1}{2}\log_2 \dfrac{\eps_0}{h^2}$, if we assume
\begin{equation}
\label{eq:smallnessh}
    h\lesssim c_\rho^{\frac{1}{(1-\rho)}} \left(\dfrac{T}{ \left(1+\|V\|_{L^\infty}^{\frac{2}{3}}\right)}\right)^{\frac{1}{(1-\rho)}}
\end{equation}
then $j_*<J_h$.

Next, the derivative of $g$ is
\begin{equation*}
    g'(x)=6\kappa \left(1+\|V\|_{L^\infty}^{\frac{2}{3}}\right) -(2-\rho)c_\rho Tx^{1-\rho}.
\end{equation*}
Then 
\begin{equation*}
    x_2=\left( \frac{6\kappa \left(1+\|V\|_{L^\infty}^{\frac{2}{3}}\right)}{(2-\rho)c_\rho T }\right)^{\frac{1}{1-\rho}}.
\end{equation*}
is the unique positive root of $g'(x)=0$, so that $g$ is increasing from $0$ to $x_2$ and decreasing afterwards
{\it i.e.} $g$ has its only local maximum at $x_2$. Notice also that $x_2<x_1$ since $0<\rho<1$.

In particular,
$$
g(2^j)\leq g(x_2)=\frac{1-\rho}{(2-\rho)^{\frac{2-\rho}{1-\rho}}}
\frac{\left(6\kappa \left(1+\|V\|_{L^\infty}^{\frac{2}{3}}\right)\right)^{\frac{2-\rho}{1-\rho}}}
{(c_\rho T)^{\frac{1}{1-\rho}}}
\leq A_\rho\dfrac{\left(1+\|V\|_{L^\infty}^{\frac{2}{3}}\right)^{\frac{2-\rho}{1-\rho}}}{T^{\frac{1}{1-\rho}}}.
$$
We now chose $\rho$ so that
\begin{equation*}
    \varepsilon=\frac{\rho}{1-\rho}\quad\longleftrightarrow\quad \rho=\rho_\eps:=\dfrac{\eps}{1+\eps}.
\end{equation*}
In order to simplify notation, the constants depending on $\rho$ will be indexed
by $\eps$ instead of $\rho_\eps$. For instance, we write
\begin{equation*}
    c_\varepsilon:=c_{\rho_\eps}= \dfrac{2^{\frac{\varepsilon}{1+\varepsilon}}-1 }{2(4-2^{\frac{\varepsilon}{1+\varepsilon}})}\asymp \varepsilon\quad \text{as } \varepsilon\to 0
\end{equation*}
where $c_{\rho_\varepsilon}$ is defined in \eqref{eq:def-of-crho}.
Notice also that
\begin{equation*}
    K_\varepsilon:=K_{\rho_\varepsilon}=\dfrac{7\kappa}{c_\varepsilon}\asymp\varepsilon^{-1} \quad \text{as } \varepsilon\to 0
\end{equation*}
and that \eqref{x1-rho} write
\begin{equation}
    x_1=\left(\dfrac{6\kappa \left(1+\|V\|_{L^\infty}^{\frac{2}{3}}\right)}{c_\eps T}\right)^{1+\eps}.
    \label{x1-rho-new}
\end{equation}

Thus, the upper bound \eqref{eq:smallnessh} can be written as
\begin{equation*}
    h\lesssim c_\varepsilon^{1+\varepsilon} \left(\dfrac{T}{ \left(1+\|V\|_{L^\infty}^{\frac{2}{3}}\right)}\right)^{1+\varepsilon}.
\end{equation*}

Then we obtain
\begin{equation*}
    g(2^j)\le A_\varepsilon\dfrac{\left(1+\|V\|_{L^\infty}^{\frac{2}{3}}\right)^{2+\varepsilon}}{T^{1+\varepsilon}}
\end{equation*}
where
\begin{equation*}
    A_\varepsilon:=\frac{(1+\varepsilon)^{1+\varepsilon}}{(2+\varepsilon)^{2+\varepsilon}}\cdot\frac{1}{c_\varepsilon^{1+\varepsilon}}\asymp \varepsilon^{-1} \quad \text{as }\varepsilon\to 0. 
\end{equation*}
But then \eqref{eq:estcontcostsifull} implies
\begin{equation}
\label{eq:badestimatef}
\|f\|^2_{L^2((a_j,a_j+T_j); \ell^2(h\Z^d))}
\leq \frac{C_\varepsilon}{T}\exp\left(A_\varepsilon\dfrac{\left(1+\|V\|_{L^\infty}^{\frac{2}{3}}\right)^{2+\varepsilon}}{T^{1+\varepsilon}}\right) \|u_0\|_{\ell^2(h\Z^d)}^2.
\end{equation}

Let us now show that a better estimate is available as soon as $j>j_*$. Indeed, rewriting  \eqref{x1-rho} as $6\kappa \left(1+\|V\|_{L^\infty}^{\frac{2}{3}}\right)=c_\varepsilon Tx_1^{\frac{1}{\varepsilon+1}}$,
we get
$$
g(2^j)=6\kappa \left(1+\|V\|_{L^\infty}^{\frac{2}{3}}\right)2^j-c_\varepsilon T2^{j(1+\frac{1}{1+\varepsilon})}
=-c_\varepsilon T\bigl(2^{\frac{j}{1+\varepsilon}}-x_1^{\frac{1}{1+\varepsilon}}\bigr)2^j.
$$
But, for $j\geq j_*+1$, 
$
2^{\frac{j}{1+\varepsilon}}\geq 2^{\frac{j_*}{1+\varepsilon}}2^{\frac{1}{1+\varepsilon}}
\geq x_1^{\frac{1}{1+\varepsilon}}2^{\frac{1}{1+\varepsilon}}
$
by definition of $x_1$ and $j_*$, leading to
$$
g(2^j)\leq  -c_\varepsilon\bigl(2^{\frac{1}{1+\varepsilon}}-1\bigr)Tx_1^{\frac{1}{1+\varepsilon}}2^j
=-6\kappa \bigl(2^{\frac{1}{1+\varepsilon}}-1\bigr)\left(1+\|V\|_{L^\infty}^{\frac{2}{3}}\right)2^j:=B_\eps 2^j
$$
$$
g(2^j)\leq  -c_\varepsilon\bigl(2^{\frac{1}{1+\varepsilon}}-1\bigr)Tx_1^{\frac{1}{1+\varepsilon}}2^j
=-6\kappa \bigl(2^{\frac{1}{1+\varepsilon}}-1\bigr)\left(1+\|V\|_{L^\infty}^{\frac{2}{3}}\right)2^j
$$
with $B_\varepsilon\asymp 1$ when $\eps\to0$.
Hence from \eqref{bound-j-starnew}, for $j>j_*$ we have 
$$
g(2^j)\leq -c_\varepsilon T 2^{j_*}2^j\bigl(2^{\frac{1}{1+\varepsilon}}-1\bigr)
\lesssim -B_\varepsilon\left(\dfrac{1+\|V\|_\infty^{\frac{2}{3}}}{T}\right)^{1+\varepsilon} 2^j
$$
where
\begin{equation*}
    B_\varepsilon:=c_\eps^{-\varepsilon}\asymp 1 \quad \text{as }\varepsilon\to 0.
\end{equation*}
Then \eqref{eq:estcontcostsifull} implies
\begin{equation}
    \label{eq:betterestimatef}
\|f\|^2_{L^2((a_j,a_j+T_j); \ell^2(h\Z^d))}
\leq \frac{C}{(1-2^{-\frac{\varepsilon}{1+\varepsilon}})T}\exp\left(-B_\eps(1+\|V\|_\infty^{\frac{2}{3}}) 2^j\right)\|u_0\|_{\ell^2(h\Z^d)}^2.
\end{equation}
As $j_*<J_h$, we then write
$$
\|f\|^2_{L^2((0,T); \ell^2(h\Z^d))}=\sum_{j=0}^{j*}+\sum_{j=j_*+1}^{J_h}\|f\|^2_{L^2((a_j,a_j+T_j); \ell^2(h\Z^d))}
:=S_-+S_+.
$$
We first estimate $S_-$ we use the bound \eqref{eq:badestimatef}:
\begin{equation*}
\begin{aligned}
S_-&\le \frac{C}{(1-2^{-\frac{\varepsilon}{1+\varepsilon}})T} \sum_{j=0}^{j*} \exp\left(A_\varepsilon\dfrac{\left(1+\|V\|_{L^\infty}^{\frac{2}{3}}\right)^{2+\varepsilon}}{T^{1+\varepsilon}}\right)\|u_0\|^2_{\ell^2(h\Z^d)}\\
&\lesssim \dfrac{1}{(1-2^{-\frac{\varepsilon}{1+\varepsilon}})T}(j_*+1)\exp\left(A_\varepsilon\dfrac{\left(1+\|V\|_{L^\infty}^{\frac{2}{3}}\right)^{2+\varepsilon}}{T^{1+\varepsilon}}\right)\|u_0\|^2_{\ell^2(h\Z^d)}\\
&\lesssim \dfrac{1}{(1-2^{-\frac{\varepsilon}{1+\varepsilon}})T}\exp\left(A_\varepsilon\dfrac{\left(1+\|V\|_{L^\infty}^{\frac{2}{3}}\right)^{2+\varepsilon}}{T^{1+\varepsilon}}\right)\|u_0\|^2_{\ell^2(h\Z^d)}
    \end{aligned}
\end{equation*}
where we absorb $j_*+1$ in the exponential by increasing $K_\varepsilon$. This is possible in view of \eqref{bound-j-starnew}.

For $S_+$, we use the bound \eqref{eq:betterestimatef}:
\begin{equation*}
\begin{aligned} 
S_+    &\le \frac{C}{(1-2^{-\frac{\varepsilon}{1+\varepsilon}})T} \sum_{j=j_*+1}^{+\infty}\exp\left(-B_\varepsilon(1+\|V\|_\infty^{\frac{2}{3}}) 2^j\right)
    \|u_0\|^2_{\ell^2(h\Z^d)}\\
    &\lesssim\frac{1}{(1-2^{-\frac{\varepsilon}{1+\varepsilon}})T}\exp\left(-B_\varepsilon(1+\|V\|_\infty^{\frac{2}{3}}) 2^{j_*}\right) \|u_0\|^2_{\ell^2(h\Z^d)}\\
    &\lesssim 
    \frac{1}{(1-2^{-\frac{\varepsilon}{1+\varepsilon}})T}\exp\left(-B_\varepsilon(1+\|V\|_\infty^{\frac{2}{3}}) \right) \|u_0\|^2_{\ell^2(h\Z^d)}.
    \end{aligned}
\end{equation*}
This term is negligible compared to $S_-$ and is thus absorbed in it to give the final bound.

Finally, \eqref{estimate-j-step} implies that
$$
\|u(a_{J_h})\|_{\ell^2(h\Z^d)}\lesssim e^{g(2^{J_h})}\|u_0\|_{\ell^2(h\Z^d)}
\lesssim \|u_0\|_{\ell^2(h\Z^d)}
$$
since $J_h\geq j_*$ so that $g(2^{J_h})\leq 0$. From
\eqref{eq:estuT} we then obtain 
\begin{equation*}
    \|u(T)\|\leq  e^{-C T/h^{2}} \|u(a_{J_h})\|_{\ell^2(h\Z^d)}\lesssim
     e^{-C T/h^{2}}\|u_0\|_{\ell^2(h\Z^d)}
\end{equation*}
which is the claimed estimate.
\end{proof}

From Theorem \ref{thm-uniform-control}, we may improve Lemma \ref{adjoint-obs-system-lma}  by HUM directly as follows.

\begin{corollary}\label{inverse-obs-crc}
Let $V\in \cc_b(\mathbb R^d)$ and let $\omega$ be equidistributed.
Then for every $0<\eps\le 1$, there exists $K_\eps,C_\eps>0$ such that the semi-discrete solution $v$ of the adjoint system
\begin{equation*}
\begin{cases}
\partial_t v+\Delta_h v -Vv=0,& \text{in } (0,T)\times h\mathbb Z^d,\\
v(T)=v_F\in \Pi_{Ch^{-2},h}\bigl(\ell^2(h\mathbb Z^d)\bigr),
\end{cases}
\end{equation*}
satisfies the uniform observability estimate
\begin{equation*}
\|v(0)\|_{\ell^2(h\mathbb Z^d)}^2
\le K_T \int_0^T \|v(t)\|_{\ell^2(\omega)}^2\,\mathrm{d}t
\end{equation*}
where
\[
K_T=
\begin{cases}
\dfrac{C_\eps}{T} & \text{if } T>K_\eps(1+\|V\|_{L^\infty}^{2/3}),
\text{ and }h\le \frac{h_0}{1+\|V\|_{L^\infty}^{2/3}};\\[6pt]
\dfrac{C_\eps}{T}\exp \left(C_\eps\left(\frac{(1+\|V\|_{L^\infty}^{2/3})^2}{T}\right)^{1+\varepsilon}\right)
& \text{if } T\le K_\eps(1+\|V\|_{L^\infty}^{2/3}),
\text{ and } h\leq h_\varepsilon\left(\frac{T}{1+\|V\|_{L^\infty}^{2/3}}\right)^{1+\varepsilon}.
\end{cases}
\]

If $V\in \cc_b^1(\mathbb R^d)$, the same statement holds with $\|V\|_{L^\infty}^{2/3}$ replaced by $\|V\|_{W^{1,\infty}}^{1/2}$.
\end{corollary}
    
Finally, from this corollary we may prove the relaxed observability inequality.

\begin{proof}[Proof of Theorem \ref{th:obsineqintro}] Taking $\mu=C_3/h^2$ from Theorem \ref{thm-uniform-control} and writing
the orthogonal decomposition $v=v_1+v_2$ with $v_1=\Pi_{\mu,h}v$ we obtain
    \begin{equation*}
        \begin{aligned}
            \| v(0)\|_{\ell^2(h\Z^d)}^2\le 2\|v_1(0)\|_{\ell^2(h\Z^d)}^2+2\|v_2(0)\|_{\ell^2(h\Z^d)}^2.
        \end{aligned}
    \end{equation*}
    By parabolic dissipation, we have
    \begin{equation*}
        \|v_2(0)\|_{\ell^2(h\Z^d)}^2\le e^{-CT/h^2}\|v_2(T)\|_{\ell^2(h\Z^d)}^2\le e^{-CT/h^2} \|v_F\|_{\ell^2(h\Z^d)}^2.
    \end{equation*}
    Therefore
    \begin{equation*}
        \| v(0)\|_{\ell^2(h\Z^d)}^2\le \|v_1(0)\|_{\ell^2(h\Z^d)}^2+e^{-CT/h^2}\|v_F\|_{\ell^2(h\Z^d)}^2.
    \end{equation*}
    Applying Corollary \ref{inverse-obs-crc} from time $0$ to $T/2$, one may get
    \begin{equation*}
        \begin{aligned}
            \|v_1(0)\|_{\ell^2(h\Z^d)}^2&\le K_{T/2} \int_{0}^{T/2} \|v_1(t)\|_{\ell^2(\omega)}^2\d t\\
            &\le K_{T/2} \left(\int_{0}^{T/2} \|v(t)\|_{\ell^2(\omega)}^2\d t+\int_{0}^{T/2} \|v_2(t)\|_{\ell^2(\omega)}^2\d t\right) 
        \end{aligned}
    \end{equation*}
    By parabolic dissipation again, we have
    \begin{equation*}
        \int_0^{T/2}\|v_2(t)\|^2_{\ell^2(\omega)}\d t\le \frac{T}{2}\sup_{0\le t\le T/2}\|v_2(t/2)\|^2_{\ell^2(\omega)} \le \frac{T}{2}\sup_{0\le t\le T/2}\|v_2(t)\|^2_{\ell^2(h\Z^d)}\le \frac{T}{2} e^{-CT/(2h^2)}\|v_F\|^2_{\ell^2(h\Z^d)}.  
    \end{equation*}
    Grouping the above three inequalities, we obtain the desired result.
\end{proof}

\subsection{Sign changing potentials}

So far, we only considered non-negative potentials. One may easily
extend the results at the price of (strongly) increasing the
observability constant and requiring a stronger smallness condition on $h$.
Indeed, if $V$ is a real valued bounded potential, then $\tilde V=V+\|V_-\|_\infty$ is a positive potential.
Now assume that $v$ satisfies
\begin{equation*}
\begin{cases}
\partial_t v+\Delta_h v -Vv=0,& \text{in } (0,T)\times h\mathbb Z^d,\\
v(T)=v_F\in \ell^2(h\mathbb Z^d),
\end{cases}.
\end{equation*}
A simple computation then shows that
$\tilde v(t,x)=e^{t\|V_-\|_\infty }v(t,x)$ satisfies
\begin{equation*}\begin{cases}
\partial_t \tilde v+\Delta_h \tilde v -\tilde V\tilde v=0,& \text{in } (0,T)\times h\mathbb Z^d,\\
\tilde v(T)=e^{T\|V_-\|_\infty }v_F\in \ell^2(h\mathbb Z^d).
\end{cases}
\end{equation*}
Applying Theorem \ref{th:obsineqintro} to $\tilde v$, we obtain
the uniform observability estimate
\begin{eqnarray*}
\|v(0)\|_{\ell^2(h\mathbb Z^d)}^2=\|\tilde v(0)\|_{\ell^2(h\mathbb Z^d)}^2
&\le& K_T \int_0^T \|\tilde v(t)\|_{\ell^2(\omega)}^2\,\mathrm{d}t
+ C e^{-C/h^2}\|\tilde v_F\|_{\ell^2(h\mathbb Z^d)}^2\\
&\leq&K_Te^{2T\|V_-\|_\infty}\int_0^T \|v(t)\|_{\ell^2(\omega)}^2\,\mathrm{d}t
+ C e^{-C/h^2}e^{2T\|V_-\|_\infty }\|v_F\|_{\ell^2(h\mathbb Z^d)}^2.
\end{eqnarray*}

The term $e^{2T\|V_-\|_\infty }$ can be absorbed in $e^{-C/h^2}$ by asking $h\lesssim (T\|V_-\|_\infty)^{-1/2}$
For $T\lesssim (1+\|V\|_\infty)^{3/2}$, $e^{2T\|V_-\|_\infty}$ can be absorbed into $K_T$.
This is unfortunately not possible for $T$ large.

\section{On the necessary condition for bounded potentials}\label{sec:necessary}

\subsection{Some estimates on the heat kernel in $h\Z^d$}
The heat kernel $p_h(x,y,t)$ of the discrete Laplacian $\Delta_h$ is defined by $p_{d,h}(x,y,t)=u_y(x,t)$, where $u_y$ is the solution of 
\begin{equation*}
    \begin{cases}
        \partial_t u_y =\Delta_h u_y, &(t,x)\in (0,\infty)\times h\Z^d,\\
        u_y(0,x)=h^{-d}\delta_{y}. & 
    \end{cases}
\end{equation*}
Here the function $\delta_y$ is given by $\delta_y(y)=1$ and $\delta_y(x)=0$ for all $x\neq y$. Using the discrete Fourier transform, we obtain the explicit representation
\begin{eqnarray}
\label{eq:explicitheat}
    p_{d,h}(t,x,y)&=&\frac{1}{(2\pi)^d}\int_{Q_{2\pi h^{-1}}} e^{- 2th^{-2}\sum_{j=1}^d \left(1-\cos (\xi_j h)\right)}e^{i(x-y)\cdot \xi}\d \xi\notag\\
    &=&\frac{1}{(2h\pi)^d}\int_{Q_{2\pi}} e^{- 2th^{-2}\sum_{j=1}^d \left(1-\cos (\eta_j)\right)}
e^{i h^{-1}(x-y)\cdot \eta}\d\eta
\end{eqnarray}
where $Q_{a}=[-a/2,a/2]^d$. As a first consequence, $p_{d,h}$ is a function of $x-y$,
is the product of one-dimensional heat kernels and satisfies a scaling property
with respect to the mesh size $h$: for $t>0$ and $m,n\in\Z^d$,
\begin{equation}
    p_{d,h}(t,x,y)=h^{-d}\prod_{j=1}^d\mathfrak{p}_1\bigl(h^{-2}t,h^{-1}|x_j-y_j|\bigr).\label{scale}
\end{equation}
with
$$
\mathfrak{p}_1(\tau,u)= \frac{1}{2\pi }
\int_{-\pi}^\pi e^{- 2\tau \left(1-\cos (\eta)\right)}
e^{i u\eta}\,\mbox{d}\eta\qquad\tau>0,\ u\in\Z.
$$

We will need some simple estimates about $p_h$. We start with the $\ell^2(h\Z)$ norm:

\begin{lemma}\label{lem:ell2normheat}
    Let $0<T_-<T_+$. Then there exists $h_0=h_0(T_-,T_+)>0$
    and $\kappa=\kappa(d,T_-,T_+)>0$ such that, for $T_-\leq t\leq T_+$ and $0<h<h_0$, and every $y\in h\Z^d$,
    $$
    \dfrac{1}{\kappa}\leq \|h^{d/2}p_{d,h}(t,\cdot,y)\|_{\ell^2(h\Z^d)}\leq\kappa.
    $$
\end{lemma}

\begin{proof}
From \eqref{scale}, 
$$
\|p_{d,h}(t,\cdot,y)\|_{\ell^2(h\Z^d)}^2=h^{-2d}\|\mathfrak{p}_1(h^{-2}t,u)\|_{\ell^2(\Z)}^{2d}.
$$
Parseval then gives
$$
\|\mathfrak{p}_1(\tau,u)\|_{\ell^2(\Z)}^2=\frac{1}{2\pi}\int_{-\pi}^\pi
e^{- 4\tau \left(1-\cos (\eta)\right)}\,\mbox{d}\eta
$$
where $\tau=h^{-2}t$.
It is enough to prove that this last integral is $\sim\frac{1}{\sqrt{2\pi\tau}}$
when $\tau\to+\infty$. 

This can be obtained from the standard Laplace asymptotic. As we did not find an appropriate reference,
we now prove this in detail. Define $\ffi(s)=1-\cos(s)$ so that we want to estimate
$$
F(\tau)=\frac{1}{2\pi}\int_0^{\pi}e^{-4\tau\ffi(s)}\,\mbox{d}s
=\frac{1}{2\pi}\int_0^{\pi/2} e^{-4\tau \ffi(s)} \,\mbox{d}s +\frac{1}{2\pi}\int_{\pi/2}^\pi e^{-4\tau\ffi(s)}\,\mbox{d}s=F_0(\tau)+F_1(\tau).
$$

On $[\pi/2,\pi]$, $1-\cos(s)\geq 1$ so that
$$
F_1(\tau)\leq \frac{1}{2\pi}\int_{\pi/2}^\pi e^{-4\tau}\,\mbox{d}s=\dfrac{1}{4}e^{-4\tau}.
$$

To estimate $F_0$, note that $\ffi'(s)=\sin (s)$ vanishes at the end-point $s=0$.
Further $\ffi(s)\geq 0$ so we write
$\ffi(s)=\chi(s)^2$ so that $\chi$ is smooth over $]0,\pi/2]$ with
$$
\chi'(s)=\dfrac{\sin(s)}{2\sqrt{1-\cos (s)}}
\to \frac{1}{\sqrt{2}}
$$
when $s\to 0$. In particular, $\chi$ extends into
a smooth function at $0$ with $\chi'(0)=\frac{1}{\sqrt{2}}$. From the above, it is clear
that $\chi$ is a diffeomorphism $[0,\pi/2]\to[0,1]$. But then, changing variable 
$t=\chi(s)$ gives
$$
F_0(\tau)=\frac{1}{2\pi}\int_0^{\pi/2}e^{-4\tau\chi(s)^2}\,\mbox{d}s
=\frac{1}{2\pi}\int_0^1e^{-4\tau t^2}\,\frac{\mbox{d}t}{\chi'\bigl(\chi^{-1}(t)\bigr)}.
$$
Write $\psi(t)=\mathbf{1}_{[0,1]}\dfrac{1}{\chi'\bigl(\chi^{-1}(t)\bigr)}$ (which is bounded)
and change variable $u=2\sqrt{\tau} t$ to obtain
$$
F_0(\tau)=\frac{1}{4\pi \tau^{1/2}}\int_{\R}\psi\left(\frac{u}{\sqrt{2\tau}}\right)e^{-u^2}\,\mbox{d}u.
$$
From dominated convergence, when $\tau\to \infty$,
$$
F_0(\tau)\to \frac{\psi(0)}{4\pi \tau^{1/2}}\int_{\R}e^{-z^2}\,\mbox{d}z
=\frac{\sqrt{2\pi}}{4\pi \tau^{1/2}}
$$
as claimed.
\end{proof}

The second estimate we need is a pointwise bound.
As $\mathfrak{p}_1$ can be expressed in terms of Bessel functions such an estimate can be obtained from classical
estimates for those functions. A more direct approach was given by  Pang \cite{pang}. 
We will now simplify his estimates in order to obtain bounds that are easier to handle for our purpuse:

\begin{lemma}
For all $t>0$ and $x,y\in \Z^d$,
\begin{equation}
    \label{eq:upper-bound}  
 p_{d,h}(t,hx,hy) \lesssim t^{-d/2}\exp\left(-\sum_{j=1}^d \frac{1}{2}|x_j-y_j|\ln\left( 1+\frac{h^2|x_j-y_j|}{2t}\right)\right).   
\end{equation}
Further, for every $L>0$, $T_+>T_->0$ there exists $\mu$
such that, if $hx-hy\in [-L,L]^d$ and $T_-<t<T_+$, then 
\begin{equation}
    \label{eq:lower-bound}  
 p_{d,h}(t,hx,hy) \geq\mu \exp\left(-\sum_{j=1}^d \frac{1}{2}|x_j-y_j|\ln\left( 1+\frac{h^2|x_j-y_j|}{2t}\right)\right)
\end{equation}    
\end{lemma}

Note that the upper bound implies that $p_{h,d}(t,x,y)$ is bounded independently of $h$ and $t\geq 1$.

\begin{proof}
Let us first introduce
$$
\zeta(s)=\mathrm{arcsinh}(s)+\frac{1-\sqrt{s^2+1}}{s}=\ln(s+\sqrt{s^2+1})-\dfrac{s}{1+\sqrt{s^2+1}}
$$
and observe that 
$$
\dfrac{1}{2}\ln(1+s)\leq \zeta(s)\leq \ln(1+s).
$$
Indeed, consider $\zeta_+(s)=\ln(1+s)-\zeta(s)$ then $\zeta_+(0)=0$ and
$$
\zeta_+^\prime(s)=\frac{1}{s+1}-\frac{1}{\sqrt{s^2+1}+1}\geq 0
$$
showing that $\zeta_+\geq 0$. On the other hand, if we consider $\zeta_-(s)=\zeta(s)-\dfrac{1}{2}\ln(1+s)$,
then $\zeta_-(0)=0$ and
$$
\zeta_-^\prime(s)=\frac{s(2\sqrt{1+s^2}+2-s)}{2(1+s)(1+\sqrt{1+s^2})^2}\geq 0
$$
showing that $\zeta_-\geq 0$ as well.
Then \cite[Theorem 3.5]{pang} states that
$$
\min(|u|^{-1/2},\tau^{-1/2})\exp\left(-|u|\zeta\left(\frac{|u|}{2\tau}\right)\right)\lesssim
\mathfrak{p}_1(\tau,u)\lesssim \min(|u|^{-1/2},\tau^{-1/2})\exp\left(-|u|\zeta\left(\frac{|u|}{2\tau}\right)\right).
$$
Combining this with \eqref{scale}, we obtain the desired bounds.
\end{proof}

\begin{corollary}
\label{lem:remainderheat}
Let $L>1$ and $T_+>T_->0$. Then there exists $h_1=h_1(T_-,T_+,L)$, $\nu=\nu(T_+)>0$
and $\gamma=\gamma(T_-,T_+)$ such that for $T_-<t<T_+$, $0<h<h_1$ and any fixed $y\in h\Z^d$,
$$
\sum_{\substack{x\in h\Z^d\\ x-y\notin[-L,L]^d}}|h^{d/2}p_{d,h}(t,x,y)|^2\leq \gamma e^{-\nu L^2}.
$$
\end{corollary}

\begin{proof} 
We will prove a bit more precise result.

Write $y=(y',y_d)$ with $y\in h\Z^{d-1}$ and $y_d\in \Z$. We first use that $h\Z^d\setminus[-L,L]^d$ can be covered by the $d$ bands $h\Z^{d-1}\times(h\Z\setminus[-L,L])$,
$h\Z^{d-2}\times (h\Z\setminus[-L,L])\times h\Z$,..., $(h\Z\setminus[-L,L])\times h\Z^{d-1}$ 
and \eqref{scale} to reduce the problem to the case $d=1$ since
$$
\sum_{\substack{x\in h\Z^d\\ x-y\notin[-L,L]^d}}|h^{d/2}p_{d,h}(t,x,y)|^2\leq d
\|h^{(d-1)/2}p_{d-1,h}(t,\cdot,y')\|_{\ell^2(h\Z^{d-1})}^2
\left(\sum_{\substack{x\in h\Z\\ x-y_d\notin[-L,L]}}|h^{1/2}p_{1,h}(t,x,y_d)|^2\right).
$$
From Lemma \ref{lem:ell2normheat}, $\|h^{(d-1)/2}p_{d-1,h}(t,\cdot,y')\|_{\ell^2(h\Z^{d-1})}^2\lesssim 1$
so that the result follows directly from the case $d=1$.

Next, using the bound \eqref{eq:upper-bound} and setting $u=x-y$ we obtain
\begin{eqnarray*}
\sum_{\substack{x,y\in \Z\\ hx-hy\notin[-L,L]}}|h^{1/2}p_{1,h}(t,hx,hy)|^2&\lesssim&
\frac{1}{T_-}\sum_{u\in \Z\setminus[-L/h,L/h]}h\exp\left(-|u|\ln\left(1+\frac{h^2|u|}{2T_+}\right)\right)\\
&\lesssim&\int_{L/h}^{+\infty}h\exp\left(-u\ln\left(1+\frac{h^2u}{2T_+}\right)\right)\,\mbox{d}u\\
&\lesssim&\int_{L/h}^{+\infty}h\exp\left(-u\ln\left(1+\frac{hL}{2T_+}\right)\right)\,\mbox{d}u\\
&\lesssim&\frac{h}{\ln\left(1+\frac{hL}{2T_+}\right)}\exp\left(-\frac{L}{h}\ln\left(1+\frac{hL}{2T_+}\right)\right).
\end{eqnarray*}

Now set $a=\dfrac{L}{2T_+}$ and $\ffi(h)=\dfrac{h}{\ln(1+ah)}$ so that $\ffi'(h)=\dfrac{\psi(h)}{(1+ah)\ln^2(1+ah)}$
with $\psi(h)=(1+ah)\ln(1+ah)-ah$. Then $\psi(0)=0$ and $\psi'(h)=a\ln(1+ah)\geq 0$ thus $\psi\geq0$
thus $\ffi$ is increasing. In particular, as $h\leq 1$, $\ffi(h)\leq\dfrac{1}{\ln(1+a)}$.
Further, by concavity, $\ln(1+ah)\geq\dfrac{h}{\ln(1+a)}$ when $0\leq h\leq 1$ thus
$\exp\left(-\dfrac{L}{h}\ln(1+ah)\right)\leq \exp\left(-\dfrac{L}{\ln(1+a)}\right)$.

On the other hand, $\ln(1+ah)\geq \dfrac{ah}{2}$ if $ah<2.5$ (say). Grouping all estimates, we obtain
$$
\sum_{\substack{x,y\in \Z\\ hx-hy\notin[-L,L]}}|h^{1/2}p_{1,h}(t,hx,hy)|^2
\lesssim
\begin{cases}
\dfrac{1}{\ln(1+L/2T_+)}\exp\left(-\dfrac{L}{\ln(1+L/2T_+)}\right)&\mbox{for }h\leq 1,\\[12pt]
\dfrac{2T_+}{L}\exp\left(-\frac{L^2}{4T_+}\right)
&\mbox{if }h<\dfrac{4T_+}{L}.
\end{cases}
$$
This gives the desired bound.
\end{proof}

Now let $V\in \cc_b(\R^d)$ be a bounded continuous function on $\R^d$.
Denote by $p_{V,h}(x,y,t)$ the heat kernel of the Schrödinger operator $P_h=-\Delta_h+V$ with $V\in \cc_b(\R^d)$. By the Feynman--Kac formula (see, e.g., \cite[Sec.~2.5]{Keller2021graphs}), we have the estimate
$$
e^{-t\|V\|_{L^\infty}}p_h(t,x,y)\le p_{V,h}(t,x,y)\le e^{t\|V\|_{L^\infty}} p_h(t,x,y).
$$

\subsection{A necessary condition on the set $\omega$}
We are now in position to prove the main result of this section
which shows that the condition on $\omega$ in Theorem \ref{th:obsineqintro}
is almost optimal.

\begin{theorem}
    Let $\omega$ be an open set in $\R^d$ with smooth boundary, 
    and $V\in\cc_b(\R^d)$. Assume that there exists $T>0$, $h_0=h_0(\omega,V,T)$ and $C=C(\omega,V,T)$ such that, for every $h\leq h_0$ and every $u_F\in\ell^2(h\Z^d)$, the solution $u$ of
    \begin{equation}
    \label{eq:heatend}
    \begin{cases}
        \partial_t u =-\Delta_h u+Vu, &(t,x)\in (0,\infty)\times h\Z^d,\\
        u(T,x)=u_F(x) &  x\in h\Z^d
    \end{cases}
\end{equation}
satisfies the following observability inequality,
\begin{equation}
    \|u(0,\cdot)\|_{\ell^2(h\Z^d)}^2\le C \int_0^T \| u(t,\cdot)\|_{\ell^2(\omega)}^2 \d t + Ce^{-C/h^2} \|u(T)\|_{\ell^2(h\Z^d)}^2.
\end{equation}
Then $\omega$ is thick.
\end{theorem}

\begin{proof}
First changing $t$ into $T-t$, we may consider $u$ that satisfies
    \begin{equation}
    \label{eq:heatbegin}
    \begin{cases}
        \partial_t u =\Delta_h u-Vu, &(t,x)\in (0,T)\times h\Z^d,\\
        u(0,x)=u_0(x) &  x\in h\Z^d
    \end{cases}
\end{equation}
and the observability inequality,
\begin{equation}
    \|u(T,\cdot)\|_{\ell^2(h\Z^d)}^2\le C \int_0^T \| u(t,\cdot)\|_{\ell^2(\omega)}^2 \d t + Ce^{-C/h^2} \|u(0)\|_{\ell^2(h\Z^d)}^2.\label{b-3}
\end{equation}

Fix an arbitrary point $x_0\in h\Z^d$ and take
\begin{equation*}
    u_0(x)=h^{d/2}p_{V,h}(1,x,x_0),\quad x\in h\Z^d
\end{equation*}
and $u$ the corresponding solution of \eqref{eq:heatbegin}.

By the semigroup property, one has
\begin{equation}
    p_{V,h}(t,x,y)=h^d\sum_{z}p_{V,h}(s,x,z)p_{V,h}(t-s,z,y),\quad \forall x,y\in h\Z^d,\, 0<s<t.\label{b-2}
\end{equation}
so that we obtain the solution 
\begin{equation*}
    u(t,x)=h^{d/2}p_{V,h}(t+1,x,x_0),\quad (t,x)\in (0,\infty)\times h\Z^d.
\end{equation*}

Let us now bound each term appearing in \eqref{b-3}.
First, if $h\leq h_0$ from Lemma \ref{lem:ell2normheat}, then, with $\kappa$ given by that lemma,
$$
\|u(T,\cdot)\|_{\ell^2(h\Z^d)}^2\gtrsim e^{-2(T+1)\|V\|_\infty}\|h^{d/2}p_{h,d}(T+1,x_0,\cdot)\|_{\ell^2(h\Z^d)}^2
\geq \frac{e^{-2(T+1)\|V\|_\infty}}{\kappa^2}.
$$
On the other hand, still from Lemma \ref{lem:ell2normheat} we get
$$
Ce^{-C/h^2} \|u_0\|_{\ell^2(h\Z^d)}^2\leq Ce^{-C/h^2}\kappa^2e^{2(T+1)\|V\|_\infty}
\leq \frac{e^{-2(T+1)\|V\|_\infty}}{2\kappa^2}
$$
if $h$ is small enough. Up to replacing $h_0$ by a smaller one, we may thus assume that this holds.
Setting $\tilde\kappa=\sqrt{2/C}e^{2(T+1)\|V\|_\infty}\kappa$, \eqref{b-3} implies that
\begin{equation}
    \label{eq:conterexampleend}
\int_0^T \| u(t,\cdot)\|_{\ell^2(\omega)}^2 \d t\geq \frac{1}{\tilde\kappa^2}.
\end{equation}

Now we take $L>0$, $T_-=1$, $T_+=T+1$.
Next, we apply Corollary \ref{lem:remainderheat}: let $h_2=\min (h_0,h_1)$ ($h_1$ given by the corollary)
there are $\gamma,\nu$ (independent of $x_0$) such that, for $t\in(0,T)$ and $h\leq h_2$,
$$
\| u(t,\cdot)\|_{\ell^2(\omega\cap h\Z^d\setminus Q_L(x_0))}^2\leq
\| u(t,\cdot)\|_{\ell^2(h\Z^d\setminus Q_L(x_0))}^2\leq \gamma e^{-\nu L^2}\leq \frac{1}{2T\tilde\kappa^2}
$$
if $L$ is large enough. Then \eqref{eq:conterexampleend} implies
$$
\int_0^T \| u(t,\cdot)\|_{\ell^2(\omega\cap Q_L(x_0))}^2 \d t\geq \frac{1}{2\tilde\kappa^2}.
$$
Now we use the crude bound: $p_{h,d}(t,x,y)$ is bounded independently of $h$ and $t\geq 1$.
It follows that there is a $\mu=\mu(T,V)$ such that $p_{V,h}\leq \mu$ for $1\leq t\leq T+1$.
From this, we deduce that
$$
T\mu^2 h^d|\omega\cap Q_L(x_0)|\geq \frac{1}{2\tilde\kappa^2}
$$
for all $h$ small enough. Letting $h\to 0$ we obtain that $|\omega\cap Q_L(x_0)|\geq \frac{1}{2\tilde\kappa^2T\mu^2}$. Since every $x_0\in \R^d$ can be approximated by points $x_h\in h\Z^d$ as $h\to 0$ we obtain
the desired property.
\end{proof}

\section*{Acknowledgements}

The authors acknowledge the use of ChatGPT on their written text in order to polish the spelling, grammar, and general style.

This work was started during a research stay of Yunlei Wang at the University of the Basque Country, funded by Euskampus through LTC Transmath. He wishes to thank
this institution for its hospitality and financial support.

The authors wish to thank Aingeru Fern\'andez-Bertolin for valuable conversations and careful reading of the manuscript.

\smallskip

Y. Bourroux has benefited from state support managed by the Agence
Nationale de la Recherche (French National Research Agency) under reference ANR-20-SFRI-0001. This research is also supported by the Spanish Agencia Estatal de Investigaci\'on, through Grant PID2024-156267NB-I00 funded by MICIU/AEI/10.13039/501100011033 and cofunded by the European Union.

The authors were supported by the French National
Research Agency (ANR) under contract number ANR-24-CE40-5470.

\end{document}